\documentclass[a4paper, 12pt, twoside,openright]{article}

\usepackage{etex} 

\usepackage{enumerate}
\usepackage{enumitem}

\usepackage[section]{placeins}	
\usepackage{epigraph}
\usepackage[font={small}]{caption}
\usepackage{appendix}

\usepackage{etoolbox} 
\patchcmd{\thebibliography}
  {\settowidth}
  {\setlength{\itemsep}{-2pt plus 0.1pt}\settowidth}
  {}{}
\apptocmd{\thebibliography}
  {\small}
  {}{}

\usepackage{amsmath,amssymb,mathrsfs}
\usepackage{amsfonts}			
\usepackage{nicefrac}			

\usepackage{bbm} 

\usepackage{stmaryrd}
\usepackage{calc,ifthen,xspace}

\usepackage[all,cmtip]{xy}
\usepackage{array}
\usepackage{multirow}
\usepackage{booktabs}
\usepackage{colortbl}
\usepackage{tabularx}
\usepackage{multirow}
\usepackage{threeparttable}

\newcolumntype{C}{>$c<$}

\usepackage{graphicx}			
\usepackage{subcaption}
\usepackage{pdfpages}
\usepackage{rotating}
\usepackage{pgfplots}
	\usepgfplotslibrary{groupplots}
	
\usepackage{tikz}
	\usetikzlibrary{backgrounds,automata}
	\pgfplotsset{width=7cm,compat=1.3}
	\pgfplotsset{every linear axis/.append style={
		/pgf/number format/.cd,
		use comma,
		1000 sep={\,},
	}}
\usepackage{eso-pic}
\usepackage{import}

\usepackage{forest}
\usetikzlibrary{arrows.meta}

\usepackage{tikz-cd}
\usetikzlibrary{knots}

\usepackage{xspace}
\usepackage{textcomp}
\usepackage{array}
\usepackage{hyphenat}

\usepackage{tocloft}
\usepackage[pdftex,pdfborder={0 0 0},
			colorlinks=true,
			linkcolor=blue,
			citecolor=red,
			pagebackref=true,
			]{hyperref}	

\usepackage[top=2.5cm, bottom=2cm, left=3cm, right=2.5cm,
			headheight=15pt]{geometry}
\usepackage{cleveref}
\crefformat{section}{\S#2#1#3} 
\crefformat{subsection}{\S#2#1#3}
\crefformat{subsubsection}{\S#2#1#3}

\AtBeginDocument{}

\renewcommand{\leq}{\ensuremath{\leqslant}} 
\renewcommand{\geq}{\ensuremath{\geqslant}}

\DeclareMathOperator{\Z}{\mathbb Z}

\DeclareMathOperator{\Lie}{\mathcal L}
\DeclareMathOperator{\Der}{Der}

\DeclareMathOperator{\rk}{rk}

\DeclareMathOperator{\Aut}{Aut}

\DeclareMathOperator{\gr}{gr}

\DeclareMathOperator{\Ann}{Ann}
\DeclareMathOperator{\Lynd}{Lynd}

\usepackage{amsthm}

\makeatletter
\newtheorem*{rep@theorem}{\rep@title}
\newcommand{\newreptheorem}[2]{%
\newenvironment{rep#1}[1]{%
 \def\rep@title{#2 \ref{##1}}%
 \begin{rep@theorem}}%
 {\end{rep@theorem}}}
\makeatother

\theoremstyle{plain}
\newtheorem{theo}{Theorem}
\newreptheorem{theo}{Theorem}
\newtheorem{fait}[theo]{Fact} 
\newtheorem{lem}[theo]{Lemma}
\newtheorem{prop}[theo]{Proposition}
\newreptheorem{prop}{Proposition}

\newtheorem{cor}[theo]{Corollary}
\newreptheorem{cor}{Corollary}

\newcounter{PB}
\newtheorem{pb}[PB]{Problem}
\newreptheorem{pb}{Problem}

\numberwithin{equation}{theo}

\theoremstyle{definition}
\newtheorem{defi}[theo]{Definition}
\newtheorem{nota}[theo]{Notation}
\newtheorem{conv}[theo]{Convention}

\newtheorem{rmq}[theo]{Remark}

\renewcommand{\k}{\Bbbk}

\newcommand{\Enstq}[2]{\left\{\ #1\ \middle|\ #2\ \right\}}    

\newcommand*{\longhookrightarrow}{\ensuremath{\lhook\joinrel\relbar\joinrel\rightarrow}}

\newenvironment{relations}
 {\enumerate[itemsep=-1pt, label=\textbf{(R\arabic*)}, ref=R\arabic*]}
 {\endenumerate}

\let\oldpagenumbering\pagenumbering
\renewcommand{\pagenumbering}[1]{%
	\cleardoublepage
	\oldpagenumbering{#1}
}

\author{Jacques \scshape{Darn\'e}}
\title{Milnor invariants of braids and welded braids \newline up to homotopy}
\date{April 1st, 2020}

\hypersetup{
    pdfauthor={Jacques Darn\'e},
    pdfsubject={Preprint},
    pdftitle={Milnor homotopy invariants for braids and welded braids},
    pdfkeywords={algebraic topology, group theory, automorphisms of free groups}
}

\begin{document}

\maketitle

\begin{abstract}
We consider the group of pure welded braids (also known as loop braids) up to (link-)homotopy. The pure welded braid group classically identifies, \emph{via} the Artin action, with the group of basis-conjugating automorphisms of the free group, also known as the McCool group $P\Sigma_n$. It has been shown recently that its quotient by the homotopy relation identifies with the group $hP\Sigma_n$ of basis-conjugating automorphisms of the reduced free group. In the present paper, we describe a decomposition of this quotient as an iterated semi-direct product which allows us to solve the Andreadakis problem for this group, and to give a presentation by generators and relations. The Andreadakis equality can be understood, in this context, as a statement about Milnor invariants; a discussion of this question for classical braids up to homotopy is also included.
\end{abstract}

\numberwithin{theo}{section}

\section*{Introduction}

The present paper is a contribution to the theory of loop braids (also called \emph{welded braids}), \emph{via} the study of their finite-type invariants. Finite-type invariants were defined by Vassiliev in 1990 \cite{Vassiliev} and were much studied during the 90's (see for instance \cite{Kontsevich, Gusarov}), giving birth to a whole field of research, which is still very active nowdays. Finite-type invariants of string-links and braids have been the focus of several papers in the late 90's, by Stanford \cite{Stanford1, Stanford2}, Mostovoy and Willerton \cite{Mostovoy-Willerton}, and Habegger and Masbaum \cite{Habegger-Masbaum}. By then, finite-type invariants of braids were fairly well-understood. Meanwhile, a generalization of finite-type invariants to virtual knotted objects was introduced in \cite{GPV}. However, it was only much later that this definition was used and studied for welded knotted objects \cite{Bar-Natan-wI, Bar-Natan-wII}. 
In the meantime, the interest for welded knotted objects had grown, as the link between welded diagrams, four-dimensional topology and automorphisms of the free group had become more apparent \cite{FRR, Satoh-Virtual, Baez}; see \cite{Damiani} for a survey of welded braids. In recent years, the study of these objects has been flourishing; see for instance \cite{Kamada, Bardakov-Bellingeri, Audoux, NNSY, Meilhan-Yasuhara,  Damiani2}. In particular, link-homotopy for these objects (corresponding to self-virtualization moves in welded diagrams) has been the focus of several recent papers \cite{ABMW, AMW, Audoux-Meilhan}.

The invariants under scrutiny in this paper appear naturally as filtrations on groups. Precisely, suppose given a group $G$, whose elements are the objects one is interested in. For example, these could be mapping classes of a manifold, automorphisms of a group, (welded) braids up to isotopy, (welded) braids up to homotopy, etc. Suppose also given a filtration of $G$ by subgroups: $G = G_1 \supseteq G_2 \supseteq \cdots$. Then one can consider the class $[g]_d$ of an element $g \in G$ inside $G/G_{d+1}$ and hope to understand $g$ through its approximations $[g]_d$, which become finer and finer as $d$ grows to infinity.  These approximation are often easier to understand than $g$. For instance, $[g]_d$ could be described by a finite family of integers (or other simple mathematical objects), that we would call \emph{invariants of degree at most $d$}.

With this point of view, the question of comparing different filtrations on the same group (such as the Andreadakis problem -- see \cref{intro_to_Andreadakis}) can be interpreted as a problem of comparison between different kinds of invariants. Conversely, comparing different notions of invariants on elements of a group can often be interpreted as a problem of comparison between different filtrations on the group, provided that these invariants are indexed by some kind of degree measuring their accuracy, and that they possess some compatibility with the group structure. It is mainly the latter point of view that we adopt below, working with filtrations on groups, with a rather algebraic point of view, getting back to the language of invariants only to interpret our results. This is motivated by the fact that the invariants we consider are strongly compatible with the group structures: not only do they come from filtrations by subgroups, as described above, but these filtrations are \emph{strongly central}, a very nice property allowing us to study them using Lie algebras. Moreover, all the filtrations we consider do have a natural algebraic definition.

We consider mainly three kinds of filtrations (or invariants):
\begin{itemize}[noitemsep,topsep=3pt]
\item \textbf{Minor invariants} correspond to \emph{Andreadakis-like filtrations} (or the Johnson filtration for the Mapping Class Group). These are defined for automorphism groups of groups, and there subgroups.
\item \textbf{Finite-type (or Vassiliev) invariants} with coefficients in a fixed commutative ring $\k$ correspond to the \emph{dimension filtration} $D_*^{\k}G = G \cap (1+ I^*)$, where $I$ is  the augmentation ideal of the group ring $\k G$.
\item \textbf{The lower central series} on $G$ is the minimal strongly central filtration on $G$.
\end{itemize}
The minimality of the lower central series means that the corresponding invariants of degree $d$ contain as much information as possible for invariants possessing this compatibility with the group structure. Since the two other filtrations are also strongly central, and the Milnor invariants are of finite type, the above list goes from the coarsest invariants to the finest ones. Thus, although we will not always emphasize this in the sequel, the reader should keep in mind that a statement of the form ``Milnor invariants of degree at most $d$ distinguish classes of elements $g \in G$ modulo $\Gamma_{d+1}G$" \emph{implies} that Milnor invariants of degree at most $d$ are universal finite-type invariants of degree at most $d$, and that finite-type invariants of degree at most $d$ distinguish classes of elements $g \in G$ modulo $\Gamma_{d+1}G$.

\subsection*{Main results}

We are interested in the group of pure welded braids (or pure welded string-links) up to homotopy. This group identifies, through a version of the Artin action up to homotopy, with the group $hP\Sigma_n$ of (pure) basis-conjugating automorphisms of the \emph{reduced free group} $RF_n$ (see Def.\ \ref{def1_RF}). The key result of this paper is the decomposition theorem:
\begin{reptheo}{dec_of_hMcCool}
There is a decomposition of $hP\Sigma_n$ into a semi-direct product:
\[hP\Sigma_n \cong \left[\left( \prod\limits_{i <n} \mathcal N(x_n)/x_i\right) \rtimes (RF_n/x_n)\right] \rtimes hP\Sigma_{n-1},\]
where $\mathcal N(x_n)/x_i$ is the normal closure of $x_n$ inside $RF_n/x_i$, and the action of $RF_n/x_n \cong RF_{n-1}$ on the product is the diagonal one. Moreover, the semi-direct product on the right is an almost direct one.
\end{reptheo}
The reduced free group is studied in \cref{Section1}. In particular, using the version of the Magnus expansion for the reduced free groups introduced by Milnor, which takes values in the \emph{reduced free algebra}, we are able to show an analogue of Magnus' theorem:
\begin{reptheo}{L(RF) = LR}
The Lie ring of the reduced free group identifies with the \emph{reduced free algebra} on the same set of generators.
\end{reptheo}

The restriction $hP\Sigma_n \cap \mathcal A_*(RF_n)$ of the Andreadakis filtration $\mathcal A_*(RF_n)$ of $RF_n$ encodes Milnor invariants of pure welded braids. We are able to determine the structure of the associated graded Lie algebra in \cref{Section2}:
\begin{reptheo}{Johnson_conj_iso}
The Lie algebra $\Lie \left(hP\Sigma_n \cap \mathcal A_*(RF_n)\right)$ identifies, \emph{via} the Johnson morphism, to the algebra of \emph{tangential derivations} of the reduced free algebra.
\end{reptheo}

On the other hand, the decomposition of $hP\Sigma_n$ (Theorem \ref{dec_of_hMcCool}) induces a decomposition of its lower central series, which in turn gives a decomposition of the associated Lie algebra (Theorem \ref{L(wPn)}). We are thus able to compare the lower central series and the Andreadakis filtrations via a comparison of their associated graded Lie algebras, getting the promised comparison result, which we also show for classical braids up to homotopy:
\begin{reptheo}{Andreadakis_for_wP_n}
The Andreadakis equality holds for $G = hP_n$ and $G = hP\Sigma_n$: 
\[G \cap \mathcal A_*(RF_n) = \Gamma_*G.\] 
In other words, Milnor invariants of degree at most $d$ classify braids up to homotopy (resp.\ welded braids up to homotopy) up to elements of $\Gamma_{d+1}(hP_n)$ (resp.\ $\Gamma_{d+1}(hP\Sigma_n)$).
\end{reptheo}

Remark that there is no obvious link between this theorem and its analogue up to isotopy. On the one hand, for classical braids up to isotopy, the fact that Milnor invariants can detect the lower central series has been known for a long time \cite{Mostovoy-Willerton, Habegger-Masbaum}, but the result up to homotopy is new, and cannot be deduced from the former (as far as I know). On the other hand, for welded braid (that is, for basis-conjugating automorphisms of the free group), the result up to isotopy is still opened. One feature of $hP_n$ and $hP\Sigma_n$ which makes them very different from $P_n$ and $P\Sigma_n$ (and in fact, easier to handle) is their \emph{nilpotence}, which is used throughout the paper.

\smallskip

Finally, we use our methods to give a presentation of the group $hP\Sigma_n$. A classical result of McCool \cite{McCool} asserts that the group $P\Sigma_n$ of (pure) basis-conjugating automorphisms of the free group $F_n$ is the group generated by generators  $\chi_{ij}\ (i \neq j)$ submitted to the \emph{McCool relations}:
\[\begin{cases}
[\chi_{ik}\chi_{jk}, \chi_{ij}] = 1 &\text{ for $i,j,k$ pairwise distinct,} \\
[\chi_{ik}, \chi_{jk}] = 1 &\text{ for $i,j,k$ pairwise distinct,} \\
[\chi_{ij},\chi_{kl}] = 1 &\text{ if } \{i,j\} \cap \{k,l\} = \varnothing,
\end{cases}\]
We show that we need to add three families of relation to get its quotient $hP\Sigma_n$:
\begin{reptheo}{The_presentation}
The pure loop braid group up to homotopy $hP\Sigma_n$ is the group generated by generators  $\chi_{ij}\ (i \neq j)$ submitted to the McCool relations on the $\chi_{ij}$, and the three families of relations:
\[[\chi_{mi}, w, \chi_{mi}] = [\chi_{im}, w, \chi_{jm}] = [\chi_{im}, w, \chi_{mi}]= 1,\]
for $i,j < m$, $i\neq j$, and $w \in \langle \chi_{mk} \rangle_{k < m}$.
\end{reptheo}

The method used for the group can be adapted to the Lie algebra associated to the lower central series of $hP\Sigma_n$. We show in \cref{par_pstation_Lie} that it admits a similar presentation. We also give a presentation of the Lie algebra of $hP_n$ in corollary \ref{hDrinfeld-Kohno}.

\paragraph{Acknowledgements:}
The author thanks warmly Jean-Baptiste Meilhan for having brought to his attention the group under scrutiny here, and for numerous helpful discussions about the topology involved. He thanks Sean Eberhard for his answer to the question he asked on MathOverflow about finite presentations of nilpotent groups. He also thanks Prof.\ T.\ Kohno for asking the question which lead to the results of \cref{par_pstation_Lie}.

\setcounter{tocdepth}{2}
\tableofcontents

\section{Reminders: strongly central series and Lie rings}

We give here a short introduction to the theory of strongly central filtrations and their associated Lie rings, whose foundations were laid by M. Lazard in \cite{Lazard}. Details may be found in \cite{Darne1, Darne2}.

\subsection{A very short introduction to the Andreadakis problem}\label{intro_to_Andreadakis}

Let $G$ be an arbitrary group. The left and right action of $G$ on itself by conjugation are denoted respectively by $x^y = y^{-1}xy$ and ${}^y\! x = yxy^{-1}$.
The \emph{commutator} of two elements $x$ and $y$ in $G$ is $[x,y]:= xyx^{-1}y^{-1}$.
If $A$ and $B$ are subsets of $G$, we denote by $[A, B]$ the subgroup generated by all commutators $[a,b]$ with $(a,b) \in A \times B$. 
We denote the \emph{abelianization} of $G$ by $G^{ab}:= G/[G,G]$ and its lower central series by $\Gamma_*(G)$, that is: 
\[G = \Gamma_1(G) \supseteq [G,G] = \Gamma_2(G) \supseteq [G,\Gamma_2(G)] = \Gamma_3(G)\supseteq \cdots\]

The lower central series is a fundamental example of a \emph{strongly central filtration} (or \emph{N-series}) on a group $G$:
\begin{defi}
A \emph{strongly central filtration} $G_*$ on a group $G$ is a nested sequence of subgroups  $G = G_1 \supseteq G_2 \supseteq G_3 \cdots$ such that $[G_i, G_j] \subseteq G_{i+j}$ for all $i, j \geq 1$. 
\end{defi}
In fact, the lower central series is the minimal such filtration on a given group $G$, as is easily shown by induction.

\medskip

Recall that when $G_*$ is a strongly central filtration, the quotients $\Lie_i(G_*):= G_i/G_{i+1}$ are abelian groups, and the whole graded abelian group  $\Lie(G_*):= \bigoplus{G_i/G_{i+1}}$  is a Lie ring (\emph{i.e.\ }a Lie algebra over $\Z$), where Lie brackets are induced by group commutators. The lower central series of a group is usually difficult to understand, but we are often helped by the fact that its associated Lie algebra is always generated in degree one.

\begin{conv}
If $g$ is an element of a group $G$ endowed with a (strongly central) filtration $G_*$, the \emph{degree of $g$ with respect to $G_*$} is the minimal integer $d$ such that $g \in G_d - G_{d+1}$. Since most of the filtrations we consider satisfy $\bigcap G_i = \{1\}$, this is well-defined (if not, we could just say that $d = \infty$ for elements of $\bigcap G_i$). We often speak of \emph{the class $\overline g$ of $g$ in the Lie algebra $\Lie(G_*)$}, by which we mean the only non-trivial one, in $\Lie_d(G_*) = G_d/G_{d+1}$, where $d$ is the degree of $g$ with respect to $G_*$, unless a fixed degree is specified. 
\end{conv}

When $G_*$ is a strongly central filtration on $G=G_1$, there is a universal way of defining a strongly central filtration on a group of automorphisms of $G$. Precisely, we get a strongly central filtration on a subgroup of $\Aut(G_*)$, the latter being the group of automorphisms of $G$ preserving the filtration $G_*$:
\begin{equation}\label{def_A_*}
\mathcal A_j(G_*):= \Enstq{\sigma \in \Aut(G_*)}{\forall i \geq 1,\ [\sigma, G_i] \subseteq G_{i+j}}.
\end{equation} 
The commutator is computed in $G \rtimes \Aut(G)$, which means that for $\sigma \in \Aut(G)$ and $g \in G$, $[\sigma, g] = \sigma(g)g^{-1}$. Thus, $\mathcal A_j(G_*)$ is the group of automorphisms of $G_*$ acting trivially on the quotients $G_i/G_{i+j}$ ($i \geq 1$). For instance, $\mathcal A_1(G_*)$ is the group of automorphisms of $G_*$ acting trivially on $\Lie(G_*)$. When $G_*$ is the lower central series of a group $G$, then $\Lie(G):= \Lie(\Gamma_*(G))$ is generated (as a Lie ring) by $\Lie_1(G) = G^{ab}$, so that $\mathcal A_1(G)$ identifies with the group $IA_G$ of automorphisms of $G$ acting trivially on its abelianization $G^{ab}$. Thus $\mathcal A_*(G):= \mathcal A_*(\Gamma_*(G))$ is a strongly central filtration on $IA_G$, and we can try to understand how it compares to the minimal such filtration on $IA_G$, which is its lower central series:

\begin{pb}[Andreadakis]
For a given group $G$, how close is the inclusion of $\Gamma_*(IA_G)$ into $\mathcal A_*(G)$ to be an equality ?
\end{pb}

One way to attack this problem is to restrict to subgroups of $IA_G$. Precisely, if $K \subseteq IA_G$ is a subgroup, we can consider the following  three strongly central filtrations on $K$:
\[\Gamma_*(K) \subseteq K \cap \Gamma_*(IA_G) \subseteq K \cap \mathcal A_*(G).\]

\begin{defi}
We say that the \emph{Andreadakis equality} holds for a subgroup $K$ of $IA_G$ when $\Gamma_*(K) = K \cap \mathcal A_*(G)$.
\end{defi}

Our three main tools in calculating Lie algebras are the following:

\paragraph{Lazard's theorem \cite[Th.\ 3.1]{Lazard} (see also \cite[Th.\ 1.36]{Darne1}):} if $A$ is a filtered ring (that is, $A$ is filtered by ideals $A = A_0 \supseteq A_1 \supseteq A_2 \supseteq \cdots$ such that $A_iA_j \subseteq A_{i+j}$), the subgroup $A^\times \cap (1+A_1)$ of $A^\times$ inherits a strongly central filtration $A^\times_*:= A^\times \cap (1+A_*)$ whose Lie ring embeds into the graded ring $\gr(A_*)$, \emph{via}:
\[\left\{ 
\begin{array}{clc}
\Lie\left(A^\times_*\right) &\longhookrightarrow &\gr(A_*) \\[0.2em]
\overline x                 &\longmapsto         & \overline{x-1}.
\end{array}
\right.\]
If $G$ is any group endowed with a morphism $\alpha: G \rightarrow A^\times$, then we can pull the filtration $A^\times_*$ back to $G$, and $\Lie(\alpha^{-1}(A^\times_*))$ embeds into $\Lie(A^\times_*)$, thus into $\gr(A_*)$. 

\paragraph{Semi-direct product decompositions \cite[\S 3.1]{Darne2}:} If $G_*$ is a strongly central filtration, $G_* = H_* \rtimes K_*$ is a \emph{semi-direct product of strongly central filtrations} if $G_i = H_i \rtimes K_i$ is a semi-direct product of groups for all $i$, and $[K_i, H_j] \subseteq H_{i+j}$ for all $i,j$. Then the strong centrality of $G_*$ implies that $H_*$ and $K_*$ must be strongly central. This kind of decomposition is useful because it induces a decomposition of Lie algebras:
\[\Lie(G_*) = \Lie(H_*) \rtimes \Lie(K_*).\]

Now, if $G = H \rtimes K$ is any semi-direct product of groups, then its lower central series decomposes into a semi-direct product $\Gamma_*(G) = \Gamma_*^K(H) \rtimes \Gamma_*(K)$ of strongly central filtrations, where $\Gamma_*^G(H)$ is defined by:
\[H = \Gamma_1^K(H) \supseteq [G,H] = \Gamma_2^K(H) \supseteq [G,\Gamma_2^K(H)] = \Gamma_3^K(H)\supseteq \cdots\]
When the semi-direct product is an \emph{almost-direct} one, which means that $K$ acts trivially on $H^{ab}$, then $\Gamma_*^K(H) = \Gamma_*(H)$, so that is this case:
\[\Lie(H \rtimes K) = \Lie(H) \rtimes \Lie(K).\]

\paragraph{The Johnson morphism \cite[\S 1.4]{Darne1}:} A very useful tool to study a filtration of the form $\mathcal A_*(G_*)$ is the Johnson morphism, which encodes the fact that the associated graded Lie algebra $\Lie(\mathcal A_j(G_*))$ acts faithfully on the graded Lie algebra $\Lie(G_*)$. It is defined by:
\[\tau: \left\{
\begin{array}{clc}
\Lie \left(\mathcal A_*(G_*)\right) &\longhookrightarrow &\Der\left(\Lie(G_*)\right) \\
\overline \sigma &\longmapsto &\overline{[\sigma, -]},
\end{array}
\right.\] 
which means that it is induced by $\sigma \mapsto (x \mapsto \sigma(x)x^{-1})$. Its injectivity comes from the universality of the filtration $\mathcal A_*(G_*)$.

If we want to compare the filtration $\mathcal A_*(G_*)$ with another one, we can do it using comparison morphisms. For example, if $K$ is a subgroup of $\Aut(G_*)$, the inclusion of $\Gamma_*K$ into $K \cap \mathcal A_*(G_*)$ induces a morphism $i_*: \Lie(K) \rightarrow \Lie(K \cap \mathcal A_*(G_*))$ which is injective if and only if $\Gamma_*K = K \cap \mathcal A_*(G_*)$. Thus we can show the Andreadakis equality by showing the injectivity of the morphism $\tau':= \tau \circ i_*$ ($\tau'$ is also sometimes called the Johnson morphism).

\subsection{The case of the free group}

Before beginning our study of the Andreadakis problem for the reduced free group, it may be useful to recall some basic facts about the free group case. Here $F_n$ denotes the free group on $n$ generators $x_1, ..., x_n$.

\paragraph{Magnus expansions:} The assignment $x_i \mapsto 1+X_i$ defines an embedding of $F_n$ into the group of invertible power series on $n$ non-commuting indeterminates $X_1, ..., X_n$ with integral coefficients.  In fact, it is easy to see that it defines a morphism to $1 + (X_1, ..., X_n)$, and that this induces (using universal properties) an isomorphism of completed rings:
\[\widehat{\Z F_n} \cong \widehat{T[n]},\]
where the group ring $\Z F_n$ is completed with respect to the filtration by the powers of its augmentation ideal, and the tensor algebra $T[n]$ on $n$ generators $X_1, ..., X_n$ is completed with respect to the usual valuation. One shows that the above morphism from $F_n$ to this ring is injective by showing directly that the image of a reduced non-trivial word must be non-trivial.

\paragraph{Magnus' theorem:} Using Lazard's theorem, we can get a surjection of $\Lie(F_n)$ onto the Lie ring generated in degree one inside $\gr\left(\widehat{T[n]}\right) \cong T[n]$, which is the free Lie ring $\mathfrak L[n]$ on $n$ generators. Using freeness, one shows that this surjection has to be injective as well:
\[\Lie(F_n) \cong \mathfrak L[n].\]

\paragraph{The Andreadakis problem and the Johnson morphism:} In the case of the free group, the Johnson morphism defines an embedding of $\Lie(\mathcal A_*(F_n))$ into the Lie ring of derivations of the free Lie ring.

The Andreadakis problem for automorphisms of free groups is a difficult problem. The two filtrations were first conjectured to be equal \cite[p.\  253]{Andreadakis}. This was disproved very recently \cite{Bartholdi}, but the methods used do not give a good understanding of what is going on. The Andreadakis equality is known to hold for certain well-behaved subgroups, such as the pure braid group $P_n$ \cite{Satoh_triangulaire, Darne2}. However, the problem stays largely open in general. In particular, it is open for the group $P\Sigma_n$ of basis-conjugating automorphisms (that is, for the group of pure welded braids), of which our group $hP\Sigma_n$ is a simpler version.

\section{The reduced free group and its Lie algebra}\label{Section1}

In this first section, we introduce and study the \emph{reduced free group}, which was first introduced by Milnor \cite{Milnor-LG} as the link group of the trivial link with $n$ components. Using the Magnus expansion defined in \cite{Milnor-LG}, we determine its Lie ring. 

\begin{nota}
Several of our constructions are functors on the category of sets. For such a functor $\Phi$, we denote by $\Phi[X]$ its value at a set $X$. When $X$ is finite with $n$ elements, we will often denote $\Phi[X]$ by $\Phi[n]$, or by $\Phi_n$.
\end{nota}

\subsection{The reduced free group}

\begin{defi}\label{def1_RF}
The \emph{reduced free group} on a set $X$ is the group defined by the following presentation:
\[RF[X]:= \langle X \ |\ \forall x \in X, \forall w \in F[X],\ [x,x^w] = 1 \rangle.\]
This means that it is the largest group generated by $X$ such that each element of $X$ commutes with all its conjugates.
\end{defi}

Since any $x$ commutes with itself, the relations $[x,x^w]$ of Definition~\ref{def1_RF} can also be written $[x,[x,w]]$. The next result and its proof are taken from \cite[Lem. 1.3]{Habegger-Lin}:

\begin{prop}\label{RF_is_nilp}
For any integer $n$, the group $RF_n$ is $n$-nilpotent. For any set $X$, the group $RF[X]$ is residually nilpotent. 
\end{prop}

\begin{proof}
We use the fact that $RF[-]$ is a functor on pointed sets. First, for a finite set $X$, we show by induction on $n = |X|$ that $RF_n= RF[X]$ is $n$-nilpotent. This is obvious for $n = 1$, because $RF_1 \cong \Z$. Suppose that $RF_{n-1}$ is $(n-1)$-nilpotent. If $x \in X$, the normal subgroup $\mathcal N(x)$ of $RF[X]$ generated by $x$ is the kernel of the projection $p_x$ from $RF[X]$ to $RF[X-\{x\}]$ sending $x$ to $1$. We have an exact sequence:
\[\begin{tikzcd} 1 \ar[r] &\bigcap\limits_{x \in X} \mathcal N(x) \ar[r] &RF[X] \ar[r, "p = (p_x)"] &\prod\limits_{x \in X} RF[X-\{x\}]. \end{tikzcd}\]
Since the group on the right is $(n-1)$-nilpotent by induction hypothesis, the morphism $p$ must send $\Gamma_n(RF[X])$ to $1$, so that $\Gamma_n(RF[X])$ is inside the kernel $\bigcap \mathcal N(x)$. Moreover, by definition of the reduced free group, for every $x \in X$, all elements of $\mathcal N(x)$ commute with $x$. Thus, an element of $\bigcap \mathcal N(x)$ commutes with all $x \in X$, so it is in the center $\mathcal Z(RF[X])$. As a conclusion, $\Gamma_n(RF[X]) \subseteq \mathcal Z(RF[X])$, which means exactly that $RF[X]$ is $n$-nilpotent.

Suppose now $X$ infinite. Let $w$ be an element of $RF[X]$. It can be written as a product of a finite number of elements of $X$ and their inverses. Denote be $W$ such a finite subset of $X$. Then $w$ is inside the image of the canonical injection $RF[W] \hookrightarrow RF[X]$, which is split by the projection from $RF[X]$ to $RF[W]$ sending $X - W$ to $1$. Since $RF[W]$ is $|W|$-nilpotent, this construction provides a nilpotent quotient of $RF[X]$ in which the image of $w$ is non-trivial, whence the residual nilpotence of $RF[X]$.
\end{proof}

\subsection{The reduced free algebra}\label{par_A[Y]}

\begin{defi}\label{def_of_A}
Let $Y$ be a set. If $s \geq 2$ is an integer, let us define $\Delta_s (Y)$ by:
\[\Delta_s (Y):= \{(y_i) \in Y^s\ |\ \exists i \neq j,\ y_i = y_j \}.\] 
The \emph{reduced free algebra} on $Y$ is the unitary associative ring defined by the following presentation:
\[A[Y]:= \langle Y \ |\ \forall s,\ \forall (y_i) \in \Delta_s (Y),\ y_1 \cdots y_s = 0 \rangle.\]
\end{defi}

For short, we often forget the mention of $Y$ when it is clear from the context, and write only $A$ for $A[Y]$.

\begin{fait}\label{basis_of_A}
The algebra $A[Y]$ is graded by the degree of monomials. As a $\Z$-module, $A[Y]$ is a direct factor of the tensor algebra $T[Y]$; a (finite) basis of $A[Y]$ is given by \emph{monomials without repetition} on the generators $y \in Y$,  which are monomials of the form $y_1 \cdots y_s$ with $(y_i) \notin \Delta_s (Y)$. 
\end{fait}

\begin{proof}
Let $R$ be the (free) $\Z$-submodule of $T[Y]$ generated by the $y_1 \cdots y_s$ such that $(y_i) \in \Delta_s (Y)$ (monomials \emph{with repetition}). This module is clearly a homogeneous ideal of $T[Y]$. As a consequence, $A = T/R$. Moreover, if we denote by $S$ the (free) $\Z$-submodule of $T$ generated by monomials without repetition, then $T = S \oplus R$ as a $\Z$-module, so that $A \cong S$.
\end{proof}

\begin{defi}\label{def_of_RL}
Let $Y$ be a set. The \emph{reduced free Lie algebra} on $Y$ is the Lie algebra defined by the following presentation:
\[R \mathfrak L[Y]:= \langle Y \ |\ \forall s,\ \forall (y_i) \in \Delta_s (Y),\ [y_1, ..., y_s] = 0 \rangle,\]
where $[y_1, ..., y_s]$ denotes $[y_1, [y_2, [ \cdots [y_{s-1}, y_s]] \cdots]$.
\end{defi}

The following result uses some of the combinatorics of the free Lie ring recalled in the appendix:

\begin{prop}\label{Prim(A)}
The Lie subalgebra of $A[Y]$ generated by $Y$ identifies with $R \mathfrak L[Y]$.
\end{prop}

\begin{proof}
We need to prove that the intersection of the ideal $R$ of relations defining  $A[Y]$ and the free Lie algebra $\mathfrak L[Y] \subset T[Y]$  is exactly the module $S$ of relations defining $R\mathfrak L[Y]$. The inclusion of $S$ into $R$ is clear: when we decompose a relation in $S$ on the basis of $TV$, only monomials with exactly the same letters appears, counting repetitions. For the converse, let us first remark that thanks to Lemma \ref{lem_letters}, $S$ is the submodule of $\mathfrak L[Y]$ generated by all Lie monomials with repetition. Let $p \neq 0$ be an element of  $R \cap \mathfrak L[Y]$, and let us consider its decomposition $p = \sum \lambda_w P_w$ on the Lyndon basis of $\mathfrak L[Y]$. Let $w$ be the smaller Lyndon word such that $\lambda_w \neq 0$. It follows from Lemma \ref{gr(P)=Id} that $\lambda_w$ must be the coefficient of $w$ in the decomposition of $p$ into a linear combination of monomials of $TV$. Since $p \in R$, the word $w$ must be with repetition, so that $P_w \in S$. Then $p -  \lambda_w P_w \in R \cap \mathfrak L[Y]$ has less terms than $p$ in its decomposition on the Lyndon basis, giving us the result by induction.
\end{proof}

\begin{rmq}\label{finite_pstation_RL}
When $Y$ is a finite set with $n$ elements, we can extract finite presentations from the above presentations. Indeed, the ideal $R$ and the Lie ideal $S$ are both generated in degrees at most $n+1$, since $R_{n+1} = T[n]_{n+1}$ and $S_{n+1} = \mathfrak L[n]_{n+1}$ (a word of length $n+1$ must possess at least a repetition). As a consequence, the relations of degree at most $n+1$ are enough do describe $A[n]$ (resp. $R\mathfrak L[n]$), and there are finitely many of them.
\end{rmq}

\begin{prop}\label{rk_of_RL}
Lyndon monomials \emph{without} repetition on the $y_i$ are a basis of $R\mathfrak L[Y]$. The rank of the degree-$k$ part $R\mathfrak L[n]_k$ of $R\mathfrak L[n]$ is $(k-1)! \binom{n}{k}$.
\end{prop}

\begin{proof}
Lemma \ref{lem_letters} implies that the module $S$ in the proof of Prop. \ref{Prim(A)} is the submodule generated by all Lyndon monomials with repetition, which are thus a basis of $S$. As a consequence, Lyndon monomials \emph{without} repetition give a basis of the quotient $R\mathfrak L[Y] = \mathfrak L[Y]/S$. 

In order to determine the ranks, we need to count Lyndon words without repetition of length $k$ in $y_1, ..., y_n$. A word without repetition is Lyndon if and only if its first letter is the smallest one. Such a word is determined by the choice of $k$ letters, and a choice of ordering of the $(k-1)$ letters left when the smallest one is removed. This gives $(k-1)! \binom{n}{k}$ such words, as announced.
\end{proof}

\begin{prop}\label{RF_to_A}
In $A[Y]^\times$, each element of $1+Y$ commutes with all its conjugates.
\end{prop}

\begin{proof}
Let $y$ be an element of $Y$. From the relation $y^2=0$, we deduce that $1+y$ is invertible, with $1-y$ as its inverse. Let $u \in A^\times$. Then $u(1+y)u^{-1} = 1 + uyu^{-1}$. Since $yAy = 0$, we can write:
\[(1+y)(1 + uyu^{-1}) = 1 + y + uyu^{-1} = (1 + uyu^{-1})(1+y),\]
which is the desired conclusion.
\end{proof}

\begin{nota}
From now on, we denote by $X$ and $Y$ two sets endowed with a bijection $X \cong Y$ that we will denote by $x_i \mapsto y_i$ (we consider both $X$ and $Y$ indexed by a bijection from a set of indices $I$). This notation will allow us to distinguish between the group-theoretic world and its algebraic counterpart. 
\end{nota}
From Proposition \ref{RF_to_A}, we get a well-defined morphism, which is an analogue of the Magnus expansion, and was introduced by Milnor \cite[\S 4]{Milnor-LG}:
\begin{equation}\label{canonical_inj}
\mu: \left\{\begin{array}{clc}
                 RF[X] &\longrightarrow &A[Y]^\times \\
                 x_i &\longmapsto &1+y_i.  
             \end{array}
      \right.
\end{equation}
From Lazard's theorem \cite[Th.\ 3.1]{Lazard} (see also \cite[Th.\ 1.36]{Darne1}), we get an associated morphism between graded Lie algebras:
\begin{equation}\label{canonical_inj_gr}
\overline \mu: \left\{\begin{array}{clcl}
                         \Lie(RF[X]) &\longrightarrow &\gr(A[Y]) &\cong A[Y] \\
                         \bar x_i &\longmapsto &y_i. &  
                     \end{array}
              \right.
\end{equation}
  
From this we deduce our first main theorem:  
\begin{theo}\label{L(RF) = LR}
The above morphism \eqref{canonical_inj_gr} induces a canonical isomorphism between the Lie algebra of the reduced free group and the reduced free algebra:
\[\Lie(RF[X]) \cong R\mathfrak L[Y].\]
\end{theo}

\begin{proof}
Since $\Lie(RF[X])$ is generated in degree $1$ \cite[Prop. 1.19]{Darne1} (that is, generated by the $\bar x_i$), the morphism \eqref{canonical_inj_gr} defines a surjection from $\Lie(RF[X])$ onto the Lie subalgebra of $A$ generated by $Y$, which is $R\mathfrak L[Y]$ (Proposition \ref{Prim(A)}). But $\Lie(RF[X])$ is a reduced Lie algebra on $X$, by which we mean that the relations on the $y_i$ defining $R\mathfrak L[Y]$ are true for the classes $\bar x_i$. Indeed, in $RF[X]$, the normal closure $\mathcal N(x)$ of a generator $x \in X$ is commutative. As a consequence, if $u$ is any element of $\mathcal N(x)$, then $[x,u] = 1$. Applying this to $u = [x_{r+1}, ..., x_s, x,w] \in \mathcal N(x)$ (where our notation for iterated commutators is the same as above for iterated brackets in Lie algebras), we see that any $[x_1, ..., x_r, x, x_{r+1}, ..., x_s, x,w]$ is trivial in the group, hence so is its class in the Lie algebra. Thus $y_i \mapsto \overline x_i$ defines a section of our surjection, which has to be injective.
\end{proof}

\begin{cor}\label{RF_in_A}
The morphism $\mu: x_i \mapsto 1+y_i$ \eqref{canonical_inj} from $RF[X]$ to $A[Y]^\times$ is injective.
\end{cor}

\begin{proof}
Let $w$ be an element of $\ker(\mu)$. If $w \neq 1$, then, by residual nilpotence of $RF[X]$ (Proposition \ref{RF_is_nilp}), there exists an integer $k$ such that $w \in \Gamma_k - \Gamma_{k+1}$. Thus, $\bar w$ is a non-trivial element of $\Lie_k(RF[X])$, sent to $0$ by $\overline \mu$. But $\overline \mu$ is an isomorphism (Theorem \ref{L(RF) = LR}), so this is not possible: our element $w$ must be trivial.
\end{proof}

\subsubsection*{Some remarks on finite presentations of nilpotent groups}

Every nilpotent group of finite type admits a finite presentation. This fact is easy to prove, by induction on the nilpotency class, using that finitely generated abelian groups are finitely presented, and that an extension of finitely presented groups is finitely presented. As a consequence, the reduced free group $RF_n$ on $x_1, ..., x_n$ must admit a finite presentation. Can we find a simple one ? Considering that we have a finite presentation of the associated Lie algebra, the problem does not seem to be difficult at first. Indeed, let $G_n$ is the group admitting the same finite presentation as $R\mathcal L_n$ (see Remark \ref{finite_pstation_RL}), where brackets are replaced by commutators. These relations are true in $RF_n$ (see the proof of Prop. \ref{L(RF) = LR}), thus there is a map $\pi$ from $G_n$ onto $RF_n$, which must induce an isomorphism at the level of Lie rings. However, we can deduce that $\pi$ is an isomorphism \emph{only if we know that both these groups are nilpotent}. Which raises the question: do the relations defining $G_n$ imply that it is nilpotent ?

Thus we are led to ask ourselves: \emph{what finite set of relation is needed to ensure that a group is nilpotent ?} This question is strongly related to the following question: \emph{can we give a simple finite presentation of the free nilpotent group of class $c$ ?} (where ``simple" is taken in some naive sense). This question is surprisingly difficult. The reader can convince himself that killing commutators of the form $[x_{i_0}, ..., x_{i_c}]$ (or even $[x_{i_0}^\pm, ..., x_{i_c}^\pm]$) does not seem to be enough, because the usual formulas of commutator calculus seem not to allow one to reduce to commutators of this particular form and length. Even killing all iterated commutators of length $c+1$ of the generators is only conjectured to be enough \cite{Sims, Jackson}. 

To get a presentation known to work in general, we must take a much larger one. For instance, one can kill all iterated commutators of the generators of length between $c+1$ and $2c$. This can be improved slightly by killing only relations of the form $[x,y]$, where $x,y$ are iterated commutators of the generators of length at most $c$, whose length add up to at least $c+1$. Indeed, all iterated commutators of length greater than $c$ can be written as a product of conjugates of iterated commutators of the generators of length greater than $c$ (by repeated use of the formulas $[a, bc] = [a,b] \cdot [a,c] \cdot [[c,a],b]$ and $[a,b^{-1}] = [b,a]^b$). And every such commutator has a sub-commutator of the given form (to see that, it can help to think of commutator words as rooted planar binary trees).

\medskip

In order to avoid these problems, and to keep our presentations simple, we will only give a presentation of $RF_n$ \emph{as a nilpotent group}, that is, we \emph{assume} that the group $G_n$ is the reasoning above is nilpotent, thus obtaining:
\begin{prop}\label{def2_RF}
The reduced free group $RF_n$ is the quotient of the free $n$-nilpotent group on $x_1, ..., x_n$ by the following finite set of relations:
\[\forall s \leq n,\ \forall (x_i) \in \Delta_s (X),\ [x_1, ..., x_s] = 1,\]
where $[x_1, ..., x_s]$ denotes $[x_1, [x_2, [ \cdots [x_{s-1}, x_s]] \cdots]$.
\end{prop}

The subtlety of this situation was not perceived in \cite{Cohen}, where it was assumed that this presentation (with $(n+1)$-commutators included) would automatically define a nilpotent group. Note that several results of the present paper give some insight on the group-theoretic results of \cite{Cohen}, which were stated only in terms of the underlying abelian groups, and become simpler when taking into account the Lie ring structure.

\subsection{Centralizers}

We will use corollary \ref{RF_in_A} to compute the centralizers of generators in $RF[X]$. First, we show a lemma about commutation relations in $A[Y]$:

\begin{lem}\label{centralizers_in_A}
Let $y \in Y$, and let $\lambda$ be an integer. Define the $\lambda$-centralizer $C_{\lambda}(y)$ of $y$ in $A[Y]$ to be:
\[C_{\lambda}(y):= \{u \in A[Y]\ |\ u y = \lambda y u\}.\]
If $\lambda \neq 1$, then $C_{\lambda}(y)$ is exactly $\langle y \rangle$. If $\lambda = 1$, then $C_{\lambda}(y) = \Z \cdot 1 \oplus \langle y \rangle$. 
As a consequence, $\Z \cdot 1 \oplus \langle y \rangle$ is the centralizer $C(y)$ of $y$. Also, $\langle y \rangle$ is the annihilator $\Ann(y)$ of $y$, and it is also the set of elements $u$ satisfying $uy = -yu$.
\end{lem}

\begin{proof}
If $u$ is an element of $\langle y \rangle$, then $uy = \lambda yu = 0$. Moreover, obviously, $1 \in C_1(y)$. This proves one inclusion.  Let us prove the converse. Let $u$ be an element of $C(y)$. Let us decompose $u$ as a sum of monomials without repetition $\sum \lambda_\alpha m_\alpha$ in $A$, and consider a monomial $m_\alpha \neq 1$ not containing $y$. Then $\lambda_\alpha$ is the coefficient of $m_\alpha y$ in $0 = uy - \lambda yu$, so it must be zero. Also, if $\mu$ is the coefficient of $1$ in $m$, then the coefficient of $y$ in $uy - \lambda yu$ is $(1-\lambda)\mu$, hence $\mu = 0$ if $\lambda \neq 1$. Thus all the monomials appearing in the decomposition of $u$ (except possibly $1$ if $\lambda = 1$) must contain $y$, so that $u$ belongs to $\langle y \rangle$ (resp. to $\Z \oplus \langle y \rangle$ if $\lambda = 1$).
\end{proof}

The next lemma is \cite[Lem.~1.10]{Habegger-Lin}:
\begin{lem}\label{centralizers_in_RF}
Let $x \in X$. Let $C(x)$ be the centralizer of $x$ in $RF[X]$. Then $C(x)$ is exactly the normal closure $\mathcal N(x)$ of $x$.
\end{lem}

\begin{proof}
The inclusion $\mathcal N(x) \subseteq C(x)$ follows from the definition of $RF[X]$. Let us prove the converse. From Corollary \ref{RF_in_A}, we know that $C(x) = C(1+y) \cap RF[X] = (\Z \oplus \langle y \rangle) \cap RF[X]$. Moreover, $RF[X] \hookrightarrow A[Y]$ takes values in $1 + \overline A[Y]$ (where $\overline A$ is the augmentation ideal of $A$, that is, the set of polynomials with no constant term). As a consequence, this intersection is  $(1 + \langle y \rangle) \cap RF[X]$. But $1 + \langle y \rangle$ is exactly the set of elements sent to $1$ by the projection $A[Y] \twoheadrightarrow A[Y-\{y\}]$. This projection induces the projection from $RF[X]$ to $RF[X-\{x\}]$, whose kernel is $\mathcal N(x)$, whence the result.
\end{proof}

\begin{lem}\label{centralizers_in_RL}
Let $y \in Y$. Let $C_{\mathfrak L}(y)$ be the centralizer of $y$ in $R\mathfrak L[Y]$. Then $C_{\mathfrak L}(y)$ is exactly the Lie ideal $\langle y \rangle$ generated by $y$.
\end{lem}

\begin{proof}
If we now denote by $\langle y \rangle_A$ the ideal generated by $y$ in $A$ (denoted by $\langle y \rangle$ above), we have that $C_{\mathfrak L}(y) = C_1(y) \cap R \mathfrak L [Y] = \langle y \rangle_A \cap R \mathfrak L [Y]$ is the submodule of $R\mathfrak L[Y]$ generated by Lie monomials in which $y$ appears, which is exactly $\langle y \rangle$.
\end{proof}

\begin{prop}\label{center_of_RF}
The center of $RF_n$ is the intersection of the $\mathcal N(x_i)$, and also coincide with $\Gamma_n(RF_n)$ ; it is free abelian of rank $(n-1)!$
\end{prop}

\begin{proof}
The inclusions $\Gamma_n(RF_n) \subseteq \bigcap \mathcal N(x_i) \subseteq \mathcal Z(RF_n)$ were already established in the proof of Proposition \ref{RF_is_nilp}. Let $w$ be a non-trivial element of $\mathcal Z(RF_n)$. Since $RF_n$ is nilpotent, $w \in \Gamma_k - \Gamma_{k+1}$ for some $k$, and $\overline w$ is a non-trivial element in the center of $\Lie(RF_n) \cong R\mathfrak L_n$ (see Theorem \ref{L(RF) = LR}). From Lemma \ref{centralizers_in_RL}, we deduce that $\overline w$ is in the Lie ideal $\langle y_1 \rangle \cap ... \cap \langle y_n \rangle$. As a consequence, all $y_i$ appear at least once in each Lie monomial of the decomposition of $\bar w$. Thus its degree must be at least $n$, which means that $w \in \Gamma_n(RF_n)$.

Moreover, $\Gamma_n(RF_n) = \Gamma_n(RF_n)/\Gamma_{n+1}(RF_n) = \Lie_n(RF_n)$ identifies with the degree-$n$ part $R\mathfrak L[n]_n$ of $R\mathfrak L[n]$, which is free abelian of rank $(n-1)!$ by Lemma \ref{rk_of_RL}.
\end{proof}

\section{Derivations and the Johnson morphism}

In order to tackle the Andreadakis problem for $RF_n$, we need to understand the associated Johnson morphism, whose target is the algebra of derivations of the reduced free Lie algebra.

\subsection{Derivations}\label{Section2}

We begin our study of derivations by those of $A[Y]$, which are quite easy to handle.

\begin{prop}\label{derivations_of_A}
Any derivation $d$ of $A[Y]$ sends each element $y$ of $Y$ to an element of the ideal $\langle y \rangle$. Conversely, any application $d_Y: Y \rightarrow A[Y]$ sending each $y$ into $\langle y \rangle$ extends uniquely to a derivation of $A[Y]$.
\end{prop}

\begin{proof}
First, given a derivation $d$, we can apply it to the relation $y^2=0$. We get that $(dy)y + y(dy) = 0$. Thus $dy \in C_{-1}(y)$, which means that $dy \in \langle y \rangle$ by Lemma \ref{centralizers_in_A}.

Suppose now that we are given a map $d_Y: Y \rightarrow A[Y]$ sending each $y$ into $\langle y \rangle$. Then $d_Y$ extends uniquely to a derivation $d_T$ from $T[Y]$ to $A[Y]$ (the latter being a $T[Y]$-bimodule in the obvious sense) in the usual way:
\[d_T(y_{i_1} \cdots y_{i_l}):= \sum\limits_{j=1}^l y_{i_1} \cdots y_{i_{j-1}} \cdot d_Y(y_{i_j}) \cdot y_{i_{j-1}} \cdots y_{i_l}.\]
From the hypothesis on $d_Y$, we deduce that $d$ vanished on the monomials with repetition (the sum on the left being a sum of monomials with repetition in this case), so that it induces a well defined derivation $d: A[Y] \rightarrow A[Y]$ extending $d_Y$. Unicity is obvious from the fact that $Y$ generate the ring $A[Y]$.
\end{proof}

We now turn to the study of derivations of $R\mathfrak L[Y]$. We consider only derivations (strictly) increasing the degree, that is, sending $Y$ into $R\mathfrak L[Y]_{\geq 2}$, . In fact, we will mostly be concerned with homogeneous such derivations (which raise the degree by a fixed amount), but we will see that this distinction is not important for $R\mathfrak L[Y]$ (Cor.~\ref{graded_der}).

\begin{prop}\label{dy_in_(y)}
Let $d$ be a derivation of $R\mathfrak L[Y]$. Then for any $y \in Y$:
\[dy \in \langle y \rangle + \bigcap\limits_{y' \neq y}\langle y' \rangle =: J_y,\]
where $\langle y \rangle$ is the Lie ideal generated by $y$. Conversely, any map from $Y$ to $R\mathfrak L[Y]_{\geq 2}$ satisfying this condition can be extended uniquely to a derivation of $R\mathfrak L[Y]$.
\end{prop}

Let us remark that the homogeneous ideal $J_y$ differs from $\langle y \rangle$ only in degree $|Y|-1$ (in particular, only when $Y$ is finite), since the second term is generated by Lie monomials without repetition where all $y'$ appear, save possibly $y$. Moreover, one easily sees that, for $|Y| = n$, the ideal $J_y$ contains all of $R\mathfrak L[n]_{n-1}$. 

\begin{proof}[Proof of Prop. \ref{dy_in_(y)}]
For $|Y| = 2$, remark that $R\mathfrak L[Y]_2 \subset \langle y_1 \rangle \cap \langle y_2 \rangle$ and $R\mathfrak L[Y]_{\geq 3} = \{0\}$. As a consequence, any linear map raising the degree satisfies the condition and defines a derivation, so we have nothing to show.

Let us suppose that $Y$ has at least three elements. Let $d$ be a derivation of $R\mathfrak L[Y]$, and let $y \in Y$. Take $z \in Y - \{y\}$, and consider the relation $0 = d([y,z,y]) = [dy,z,y] + [y,z,dy]$. Let us decompose $dy$ as a sum of monomials in $A[Y]$. Let $m$ be a monomial which contains neither $y$ nor $z$, and let $\lambda$ be the coefficient  of $m$ in $dy$. Then the monomial $mzy$ appears with coefficient $2\lambda$ in the decomposition of $[dy,z,y] + [y,z,dy]$, so $\lambda$ must be trivial. Since this is true for any $z \neq y$, the only monomials without repetition not containing $y$ that can appear in $dy$ are the ones containing every element of $Y$ save $y$,  which are exactly the generators of $J_y$ modulo $\langle y \rangle$. This shows that $dy \in J_y$.

To show the converse, we can restrict to homogeneous maps, since any map from $Y$ to $\rightarrow R\mathfrak L[Y]_{\geq 2}$ is a sum of homogeneous ones, and a sum of derivations is a derivation. Suppose that we are given a homogeneous map $d_Y: Y \rightarrow R\mathfrak L[Y]_{\geq 2}$ sending each $y$ into $J_y$. If $d_Y$ is not of degree $|Y|-2$, this condition amounts to $d_Y(y) \in \langle y \rangle$. This Lie ideal stands inside the associative ideal $\langle y \rangle \subset A[Y]$. We can thus use Proposition \ref{derivations_of_A} to extend this map to a derivation of $A[Y]$. This derivation sends $Y$ into $R\mathfrak L[Y]$, hence it preserves $R\mathfrak L[Y] \subset A[Y]$. As a consequence, it restricts to a derivation of $R\mathfrak L[Y]$ extending $d_Y$. 

We are left to study the case when $Y$ has $n$ elements and $d_Y$ is of degree $n-2$. Then the conditions on the elements $d_Y(y)$ are empty. We can still extend $d_Y$ to a derivation from $T[Y]$ to $A[Y]$, as in the proof of Proposition \ref{derivations_of_A}, but it does not vanish on the relations defining $A[Y]$. However, the induced Lie derivation from $\mathfrak L[Y]$ to $R\mathfrak L[Y]$ does vanish on the Lie monomials with repetition. Indeed, it vanishes on all elements of degree at least $3$ (sent to $R\mathfrak L[Y]_{\geq n+1} = \{0\}$), and there are no such monomials in degree $2$, since the elements $[y,y]$ are already trivial in $\mathfrak L[Y]$. As a consequence, it induces a well-defined derivation from $R\mathfrak L[Y]$ to itself. This derivation extends $d_y$ and is the only one to do so, since $R\mathfrak L[Y]$ is generated by $Y$. 
\end{proof}

\begin{cor}\label{graded_der}
Any derivation of $R\mathfrak L_n$ is the sum of homogeneous components:
\[\Der\left(\mathfrak L_n\right) \cong \bigoplus\limits_{k \geq 1 } \Der_k\left(R\mathfrak L_n\right).\]
\end{cor}

\begin{proof}
If $d$ is such a derivation, Proposition \ref{dy_in_(y)} shows that the homogeneous components of its restriction to $Y$ extend uniquely to derivations of $R\mathfrak L_n$, whose sum coincide with $d$ on $Y$, hence everywhere. Note that it makes sense to speak of this sum, because $Y$ is finite, so that the number of non-trivial homogeneous components of $d|_Y$ is finite.
\end{proof}

The following theorem is an analogue of \cite[Prop. 2.41]{Darne1}, replacing free nilpotent groups by reduced free groups. 

\begin{theo}\label{Johnson_iso}
Let $n \geq 2$ be an integer. The Johnson morphism is an isomorphism:
\[\Lie\left(\mathcal A_*(RF_n)\right) \cong \Der(R\mathfrak L_n).\]
\end{theo}

\begin{proof}
Take $|X| = |Y| = n$. Let $d$ be a derivation of $R\mathfrak L[Y]$, of degree $k$. We need to lift it to an automorphism $\varphi$ of $RF[X]$. We first suppose that $k \neq n - 2$. Since $d(y_i) \in \langle y_i \rangle \cap R\mathfrak L_{k+1}[Y]$ (Proposition \ref{dy_in_(y)}), we can write each $d(y_i)$ as a linear combination of Lie monomials of length $k+1$ containing $y_i$. The corresponding product of brackets in $RF[X]$  lifts $d(y_i)$ to an element $w_i$ of $\Gamma_{k+1}(RF[X]) \cap \mathcal N(x_i)$. The element $w_i x_i$ belongs to $\mathcal N(x_i)$, so it commutes with all its conjugates. As a consequence, $x_i \mapsto w_i x_i$ defines an endomorphism $\varphi$ of $RF[X]$. Since $\varphi$ acts trivially on the abelianization of $RF[X]$, which is nilpotent, it is an automorphism \cite[Lem.~2.38]{Darne1}. Moreover, by construction, we have $\tau(\bar \varphi) = d$.

Suppose now that $k = n - 2$. Then $d(y_i)$ can be any element of $R\mathfrak L_{n-1}[Y]$. Choose any lift $w_i \in \Gamma_{n-1}(RF[X])$ of $d(y_i)$. Using the usual formulas of commutator calculus, we see that for any $w \in RF[X]$, $[w_i x_i, w, w_i x_i] \equiv [x_i, w, x_i] \pmod{\Gamma_{n+1}}$. Since $[x_i, w, x_i] = 1$ and $\Gamma_{n+1}(RF_n) = \{1\}$, we conclude that $[w_i x_i, w, w_i x_i]=1$, which means exactly that $w_i x_i$ commutes with all its conjugate. The same construction as in the first case then gives an automorphism $\varphi \in \mathcal A_{n-2}$ such that $\tau(\bar \varphi) = d$.
\end{proof}

\subsection{Tangential derivations}

\begin{defi}
A \emph{tangential derivation} of $R\mathfrak L[Y]$ is a derivation sending each $y \in Y$ to an element of the form $[y,w_y]$ (for some $w_y \in R\mathfrak L[Y]$).
\end{defi}

\begin{fait}
The subset $\Der_{\tau}(R\mathfrak L[Y])$ of tangential derivations is a Lie subalgebra of $\Der(R\mathfrak L[Y])$.
\end{fait}

\begin{proof}
Let $d: y \mapsto [y, w_y]$ and $d': y \mapsto [y, w_y']$. Then an elementary calculation gives:
\begin{equation}\label{[d,d']}
[d,d'](y) = [y, [w_y, w_y'] + d(w_y') - d'(w_y)],
\end{equation}
whence the result.
\end{proof}

\begin{prop}\label{tangential_der}
Let $n \geq 2$ be an integer. The Lie subalgebra of $\Der(R\mathfrak L_n)$ generated in degree $1$ is the subalgebra $\Der_{\tau}(R\mathfrak L_n)$ of tangential derivations.
\end{prop}

\begin{proof}
Consider the derivation $d_{ij}$ sending $y_i$ to $[y_i, y_j]$ and all the other $y_k$ to $0$. From Lemma \ref{dy_in_(y)}, we know that these generate the module of derivations of degree $1$. They are tangential derivations, so the Lie subalgebra they generate is inside $\Der_{\tau}(R\mathfrak L[Y])$. Consider the set $D_i$ of tangential derivations sending all $y_k$ to $0$, save the $i$-th one. Such derivations vanish on all monomials which are not in $\langle y_i \rangle$, and preserve $\langle y_i \rangle$. Since elements of $\langle y_i \rangle$ commute with $y_i$, formula \eqref{[d,d']} implies that the following map is a morphism:
\[c_i:
\left\{\begin{array}{clc}
\mathcal R\mathfrak L[Y] &\longrightarrow & D_i \\
t               &\longmapsto     & (y_i \mapsto [y_i,t])).
\end{array}  \right.\]
It is obviously surjective, so that $D_i$ is a Lie subalgebra of $\Der_{\tau}(R\mathfrak L[Y])$. Moreover, its kernel is $\langle y_i \rangle$ (Lemma \ref{centralizers_in_RL}), so that $D_i \cong R \mathfrak L[Y]/y_i$ is in fact the free reduced Lie algebra on the $c_i(y_j) = d_{ij}$ (for $j \neq i$). Since $\Der_{\tau}(R\mathfrak L[Y])$ is the (linear) finite direct sum of the $D_i$, it is indeed generated (as a Lie algebra) by the $d_{ij}$.
\end{proof}

Recall that the \emph{McCool group} $P\Sigma_X$ is the group of automorphisms of the free group $F[X]$ on a set $X$ fixing the conjugacy class of each generator $x \in X$.
\begin{defi}
The \emph{reduced McCool group} $hP\Sigma_X$ is the subgroup of $\Aut(RF[X])$ preserving the conjugacy class of each generator $x \in X$ of $RF[X]$.
\end{defi}
This group $hP\Sigma_X$ is also called $\Aut_C(RF_X)$, but we prefer to think of it as version of $P\Sigma_X$ \emph{up to homotopy} (this terminology will be explained in \cref{Section4}). When $X$ is finite, we denote its elements by $x_1, ..., x_n$ and $hP\Sigma_X$ by $hP\Sigma_n$.

\medskip

Consider the filtration $\mathcal A_*(RF_n)$ on $\Aut(RF_n)$. It restricts to a filtration $hP\Sigma_n \cap \mathcal A_*(RF_n)$ on $hP\Sigma_X$. Moreover, since $\mathcal A_*(RF_n)$ is strongly central on the subgroup $\mathcal A_1(RF_n)$ of automorphisms acting trivially on $RF_n^{ab}$, and $hP\Sigma_n \subset \mathcal A_1(RF_n)$, this induced filtration is strongly central on $hP\Sigma_n$.

\begin{theo}\label{Johnson_conj_iso}
Let $n \geq 2$ be an integer. The Johnson morphism induces an isomorphism:
\[\Lie \left(hP\Sigma_n \cap \mathcal A_*(RF_n)\right) \cong \Der_{\tau}(R\mathfrak L_n)\]
\end{theo}

\begin{proof}
Let $\varphi: x_i \mapsto x_i^{w_i}$ be a basis-conjugating automorphism belonging to $\mathcal A_k - \mathcal A_{k+1}$. Then $\tau(\varphi)(y_i) = [y_i, \bar w_i]$ (where $\bar w_i$ is the class of $w_i$ in $\Gamma_k/\Gamma_{k+1}$), so the Johnson morphism sends $\Lie(hP\Sigma_n \cap \mathcal A_*(RF_n))$ into $\Der_{\tau}$. Moreover, it is injective by Theorem \ref{Johnson_iso}, and since $\tau(\chi_{ij}) = d_{ij}$, Proposition \ref{tangential_der} implies that it is surjective.
\end{proof}

Theorem \ref{Johnson_conj_iso}, together with Proposition \ref{tangential_der}, have an interesting consequence: \textbf{the group $hP\Sigma_n$ is maximal among subgroups of $\Aut(RF_n)$ for which the Andreadakis equality can be true.} 
Indeed, let $hP\Sigma_n \subsetneq G \subseteq \Aut(RF_n)$, and consider the comparison morphism $i_*: \Lie(G) \rightarrow \Lie(G\cap \mathcal A_*)$ obtained from the inclusion of $\Gamma_*G$ into $G\cap \mathcal A_*$. On the one hand, the Lie algebra $\Lie(G\cap \mathcal A_*)$ contains $\Lie(hP\Sigma_n \cap \mathcal A_*)$, and this inclusion must be strict, otherwise we could argue as in the proof of Lemma \ref{gen_of_nilp} to show that $G = hP\Sigma_n$. On the other hand, $\Lie(G)$ is generated in degree $1$, so that $i_*(\Lie(G)) \subseteq \Lie(hP\Sigma_n \cap \mathcal A_*)$, the latter being the subalgebra of $\Lie\left(\mathcal A_*(RF_n)\right) \cong \Der(R\mathfrak L_n)$ generated by its degree one. As a consequence, $i_*$ cannot be surjective, whence the conclusion.

\medskip
Here is another consequence of these theorems:

\begin{cor}\label{generators_of_hMcCool}
The group $hP\Sigma_n$ is generated by the $\chi_{ij}$ ($i \neq j$), and $hP\Sigma_n^{ab}$ identifies with the free abelian group generated by the $\overline \chi_{ij}$.
\end{cor}

In particular, the canonical morphism from $P\Sigma_n$ to $hP\Sigma_n$ is surjective. This means that that when it comes to basis-conjugating automorphisms, all automorphisms of $RF_n$ are \emph{tame}. This is in striking contrast with the case of free nilpotent groups \cite[\S 2.6]{Darne1}. This fact is in fact obvious from the geometrical interpretation (recalled in \cref{Section4}), but we give an algebraic proof here, using much less machinery.

\begin{proof}[Proof of cor. \ref{generators_of_hMcCool}]
Thanks to Proposition \ref{tangential_der} and Theorem \ref{Johnson_conj_iso}, we know that the classes of the $\chi_{ij}$ in $\Lie(\mathcal A_* \cap hP\Sigma_n)$ generate this Lie ring. By applying Lemma \ref{gen_of_nilp} to the finite filtration $\mathcal A_* \cap hP\Sigma_n$, we deduce that the $\chi_{ij}$ generate $hP\Sigma_n$.

As a consequence, the $\overline \chi_{ij}$ generate its abelianization. Moreover, the Johnson morphism from $hP\Sigma_n^{ab}$ to $\Der_1(R\mathfrak L_n)$ sends the $\overline \chi_{ij}$ to the linearly independant elements $d_{ij}$ of $\Der_1(R\mathfrak L_n)$. Thus the $\overline \chi_{ij}$ are a basis of $hP\Sigma_n^{ab}$.
\end{proof}

We can also use the proof of Proposition \ref{tangential_der} to compute the \emph{Hirsch rank} of the nilpotent group $hP\Sigma_n$ (which is the rank of any associated Lie algebra). We recover the formula from \cite[Rk.\ 4.9]{AMW}:

\begin{cor}
The Hirsch rank of the reduced McCool group is:
\[\rk(hP\Sigma_n) = \rk(\Der_{\tau}(R\mathfrak L_n)) = n \cdot \rk(R\mathfrak L_{n-1}) = \sum\limits_{k=1}^{n-1} \frac{n!}{(n-k-1)!\cdot k}\cdot\]
\end{cor}

\begin{proof}
The first equality is a direct consequence of Theorem \ref{Johnson_conj_iso}. The second one stems from the proof of Proposition \ref{tangential_der}, where we have shown that $\Der_{\tau}(R\mathfrak L_n)$ is (linearly) a direct sum of $n$ copies $D_i$ of $R\mathfrak L_{n-1}$. The last one is a direct application of Proposition \ref{rk_of_RL}.
\end{proof}

\section{The Andreadakis problem}\label{section_Andreadakis}

The McCool group $P\Sigma_n$ is generated by the elements $\chi_{ij}: x_i \mapsto x_i^{x_j}$ ($\chi_{ij}$ fixed all the other $x_t$). The following relations, called the \emph{McCool relations}, are known to define a presentation of the McCool group $P\Sigma_n$ \cite{McCool}. The reader can easily check that they are satisfied in $P\Sigma_n$: 
\[\begin{cases}
[\chi_{ik}\chi_{jk}, \chi_{ij}] = 1 &\text{ for $i,j,k$ pairwise distinct,} \\
[\chi_{ik}, \chi_{jk}] = 1 &\text{ for $i,j,k$ pairwise distinct,} \\
[\chi_{ij},\chi_{kl}] = 1 &\text{ if } \{i,j\} \cap \{k,l\} = \varnothing,
\end{cases}\]

Thanks to Corollary \ref{generators_of_hMcCool}, we know that $hP\Sigma_n$ is naturally a quotient of $P\Sigma_n$. We will give in \cref{par_pstation} three families of relations that need to be added to a presentation of $P\Sigma_n$ in order to get a presentation of $P\Sigma_n$. This will rely on the semi-direct product decomposition that we now describe.

\subsection{A semi-direct product decomposition}

The following decomposition theorem is the central result of the present paper. From it we will deduce the Andreadakis equality for $h P\Sigma_n$ (\cref{par_Andreadakis}) and a presentation of this group  and of its Lie ring (\cref{par_pstation}):

\begin{theo}\label{dec_of_hMcCool}
There is a decomposition of $hP\Sigma_n$ as a semi-direct product:
\[hP\Sigma_n \cong \left[\left( \prod\limits_{i <n} \mathcal N(x_n)/x_i\right) \rtimes (RF_n/x_n)\right] \rtimes hP\Sigma_{n-1},\]
where $\mathcal N(x_n)/x_i$ is the normal closure of $x_n$ inside $RF_n/x_i$, and the action of $RF_n/x_n \cong RF_{n-1}$ on the product is the diagonal one. Moreover, the semi-direct product on the right is an almost direct one.
\end{theo}

We will prove this theorem in three steps. First, we show that $hP\Sigma_n$ decomposes into a semi-direct product $\mathcal K_n \rtimes  hP\Sigma_{n-1} $. Then we investigate the structure of $\mathcal K_n$, which decomposes as $\mathcal K_n' \rtimes \mathcal RF_{n-1}$. Finally, we investigate the structure of $\mathcal K_n'$, which is abelian and decomposes as the direct product of the $\mathcal N(x_n)/x_i$.

\subsubsection*{Step 1: decomposition of $hP\Sigma_n$}

Elements of $hP\Sigma_n$ preserve the conjugacy class of $x_n$, so they preserve its normal closure $\mathcal N(x_n)$. As a consequence, any of these automorphisms induce a well-defined automorphism of $RF_n/\mathcal N(x_n) \cong RF_{n-1}$. In other words, the projection $x_n \mapsto 1$ from $RF_n$ onto $RF_{n-1}$ induces a well-defined morphism $p_n$ from $hP\Sigma_n$ to $hP\Sigma_{n-1}$. Moreover, this morphism is a split projection, a splitting $s_n$ being the map extending automorphisms by making them fix $x_n$. Let us denote by $\mathcal K_n$ the kernel of $p_n$. We thus get our first decomposition:
\begin{equation}\label{dec1}
hP\Sigma_n \cong \mathcal K_n \rtimes hP\Sigma_{n-1} 
\end{equation}
Moreover, it will follow from Lemma \ref{decomposition_lemma_1} that this is indeed an almost direct product: $\mathcal K_n^{ab}$ is generated by the classes of the $\chi_{in}$ and the $\chi_{ni}$. From corollary \ref{generators_of_hMcCool}, we know that these are sent to a linearly independant family in $hP\Sigma_n^{ab}$, so that they freely generate $\mathcal K_n^{ab}$. We thus get a direct product decomposition $hP\Sigma_n^{ab} \cong \mathcal K_n^{ab} \oplus hP\Sigma_{n-1}^{ab}$, as announced.

\subsubsection*{Step 2: structure of $\mathcal K_n$}

We first state an easy result on generators of factors in semi-direct products.

\begin{lem}\label{generating_the_kernel}
Let $G = H \rtimes K$ be a semi-direct product of groups. Suppose given a family $(h_i)$ of elements of $H$, and a family $(k_j)$ of elements of $K$ such that their reunion generate $G$. Then $K$ is generated by the $k_j$, and $H$ is generated by the $h_i^k$, for $k \in K$. 
\end{lem}

\begin{proof}
Take an element $g \in G$ and write it as a product of $h_i^{\pm 1}$ and $k_j^{\pm 1}$. Then use the formula $kh = ({}^k\! h)k$ to push the $k_j$ to the right. We obtain a decomposition $g = h'k$, where $h' \in H$ is a product of conjugates of the $h_i^{\pm 1}$ by elements of $K$, and $k \in K$ is a product of the $k_j^{\pm 1}$. This decomposition has to be the unique decomposition of $g$ into a product of an element of $H$ followed by and element of $K$. As a consequence, if $g \in H$, then $g = h'$, whereas if $g \in K$, then $g = k$, proving our claim.
\end{proof}

We can apply Lemma \ref{generating_the_kernel} to the $\chi_{ij}$ in $hP\Sigma_n \cong \mathcal K_n \rtimes hP\Sigma_{n-1}$. Indeed, the $\chi_{in}$ and the $\chi_{ni}$ are in $\mathcal K_n$, and that the other $\chi_{ij}$ belong to $hP\Sigma_{n-1}$. Hence, $\mathcal K_n$ is generated by the conjugates of the $\chi_{in}$ and the $\chi_{ni}$ by products of the other $\chi_{ij}$ and their inverses. In fact, more is true:

\begin{lem}\label{decomposition_lemma_1}
The group $\mathcal K_n$ is generated by the $\chi_{in}$ and the $\chi_{ni}$.
\end{lem}

\begin{proof}
We use the above relations to show that the subgroup $H$ of $\mathcal K_n$ generated by the  $\chi_{in}$ and the $\chi_{ni}$ is normal in $hP\Sigma_n$, that is: $[hP\Sigma_n,H] \subseteq H$.

The bracket $[\chi_{in}, \chi_{\alpha\beta}]$ is obviously in $H$ if $\alpha = n$ or $\beta = n$. Otherwise, it is trivial, except possibly when $\alpha = i$ or $\beta = i$. In the first case (since $\chi_{n \beta}$ and $\chi_{i\beta}$ commute):
\[1 = [\chi_{in}, \chi_{n \beta} \chi_{i\beta}] = [\chi_{in}, \chi_{n \beta}] ({}^{\chi_{n \beta}}\! [\chi_{in}, \chi_{i \beta}]),\]
whence $[\chi_{in}, \chi_{i \beta}] \in H.$ In the second case:
\[1 = [\chi_{in}\chi_{\alpha n}, \chi_{\alpha i}] = ({}^{\chi_{in}}\![\chi_{\alpha n}, \chi_{\alpha i}])[\chi_{in}, \chi_{\alpha i}],\]
so, using the first case: $[\chi_{in}, \chi_{\alpha i}] \in H$.

In a similar fashion, the bracket $[\chi_{ni}, \chi_{\alpha\beta}]$ belongs to $G$ if $\alpha = n$ or $\beta = n$. Otherwise, it is trivial, except when $\alpha = i$. But in this case:
\[1 = [\chi_{ni}, \chi_{i\beta} \chi_{n \beta}] = [\chi_{ni}, \chi_{i\beta}] ({}^{\chi_{n \beta}}\! [\chi_{in}, \chi_{n \beta}]),\]
so that $[\chi_{ni}, \chi_{i \beta}] \in H$. Thus, $H$ is stable under conjugation by all generators of $hP\Sigma_n$, so it is normal in $hP\Sigma_n$.
\end{proof}

\begin{rmq}\label{Kn_in_wPn}
We have used only the McCool relations here, so the analogue of Lemma \ref{decomposition_lemma_1} is also true in $P\Sigma_n$.
\end{rmq}

By looking at how elements of $\mathcal K_n$ act on $x_n$, we get a split projection $q_n$ from $\mathcal K_n$ onto $RF_{n-1}$. Namely, if $\varphi \in \mathcal K_n$ is an automorphism sending each $x_i$ to $x_i^{w_i}$, $q_n$ sends $\varphi$ onto the class $\overline w_n \in RF_n/x_n \cong RF_{n-1}$. This is well-defined, because of Lemma \ref{centralizers_in_RF}:
\[x_n^v=x_n^w\ \Leftrightarrow\ x_n^{vw^{-1}}=1\ \Leftrightarrow\ vw^{-1} \in C(x_n) = \mathcal N(x_n)\ \Leftrightarrow\ \overline v = \overline w.\]
Moreover, this defines a morphism. Indeed, if $\varphi$ and $\psi$ send $x_n$ respectively to $x_n^{w_n}$ and $x_n^{v_n}$, then:
\[\psi \varphi (x_n) = \psi(x_n^{w_n}) = x_n^{v_n \psi(w_n)},\]
and since $\psi \in \mathcal K_n$, we have $\overline{\psi(w_n)} = \overline w_n$, whence:
\[q_n(\psi \varphi) = \overline{v_n \psi(w_n)} = \overline v_n \overline w_n = q_n(\psi)q_n(\varphi).\]
This morphism $q_n$ is a retraction of the inclusion $t_n$ of $RF_{n-1} \cong RF_n/x_n$ into $\mathcal K_n$ sending $w \in RF_n$ to the automorphism fixing all $x_i$ save $x_n$, which is sent to $x_n^w$. If we call $\mathcal K_n'$ the kernel of $q_n$, we thus get a decomposition:
\begin{equation}\label{dec2}
\mathcal K_n = \mathcal K_n' \rtimes RF_{n-1}.
\end{equation}

\begin{lem} The above decomposition is $hP\Sigma_{n-1}$-equivariant, with respect to the action of $hP\Sigma_{n-1}$ on $\mathcal K_n$ (and on $\mathcal K_n' \subset \mathcal K_n$) coming from conjugation in $hP\Sigma_n$, and to the canonical action of $hP\Sigma_{n-1}$ on $RF_{n-1}$. Precisely, $q_n$ and $t_n$ are $hP\Sigma_{n-1}$-equivariant morphisms.
\end{lem}

\begin{proof}
If $\varphi \in \mathcal K_n$ sends $x_i$ to $x_i^{w_i}$ as above, and $\chi \in hP\Sigma_{n-1}$, then $\chi \varphi \chi^{-1}$ sends $x_n$ to $x_n^{\chi(w_n)}$, so that:
\[q_n(\chi \varphi \chi^{-1}) = \overline{\chi(w_n)} = \chi(\overline{w_n}) =\chi(q_n(\varphi)).\]
As for the equivariance of $t_n$, if $w \in RF_{n-1}$, both $\chi \cdot t_n(w) \cdot \chi^{-1}$ and $t_n(\chi(w))$ fix all $x_i$ save $x_n$, the latter being sent to $x_n^{\chi(w)}$, hence they are equal.
\end{proof}

\begin{rmq}\label{dec_of_Kn_in_wPn}
A similar decomposition holds in $hP\Sigma_n$, replacing $RF_{n-1}$ by $F_{n-1}$. The same proof works, replacing the equality $C(x_n) = \mathcal N(x_n)$ (which is not true in this case) by the inclusion $C(x_n) \subset \mathcal N(x_n)$.
\end{rmq}

\subsubsection*{Step 3: structure of $\mathcal K_n'$}

So far, we have not really used the fact that that we consider welded braids \emph{up to homotopy} (that is, automorphisms of $RF_n$, not of $F_n$). In fact, the analogues of the decomposition results above are true in the group $P\Sigma_n$ of welded braids (see Remarks \ref{Kn_in_wPn} and \ref{dec_of_Kn_in_wPn}). We now come to the part where the homotopy relation plays a crucial role. That is, we are going to use the relations defining $RF_n$ in a crucial way. These relations, saying that each element $x_i$ of the fixed basis commutes with its conjugates, can be re-written as:
\[\forall i \leq n,\ \forall s,t \in RF_n,\ \ x_i^{sx_it} = x_i^{st}.\]
In other words, for $w \in RF_n$, $x_i^w$ depends only on the class of $w$ modulo $x_i$ (that is, modulo the normal closure of $x_i$). These relations allow us us to say more about the above decomposition of $\mathcal K_n$:

\begin{lem}\label{decomposition_lemma_2}
The kernel $\mathcal K_n'$ of the projection $q_n: \mathcal K_n \twoheadrightarrow RF_{n-1}$ is an abelian group, isomorphic to the product of the $\mathcal N(x_n)/x_i$, where $\mathcal N(x_n)/x_i$ is the normal closure of $x_n$ inside $RF_n/x_i \cong RF_{n-1}$. Precisely, the identification of $\mathcal N(x_n)/x_i$ with a factor of $\mathcal K_n'$ is induced by the map:
\[c_i:
\left\{\begin{array}{cll}
\mathcal N(x_n) &\longrightarrow &  \hspace*{1.9em} \mathcal K_n'\\
u               &\longmapsto     & \left(x_j \mapsto 
\begin{cases} x_i^u &\text{ if } j=i, \\ x_j &\text{ else.}
                                   \end{cases} \right)
                                   
\end{array}  \right.\]
which is a well-defined group morphism. Furthermore, $c_i$ is $RF_{n-1}$-equivariant, where $RF_{n-1} \cong \langle \chi_{nj} \rangle_j$ acts via automorphisms on the source, and via conjugation on the target.
\end{lem}

\begin{proof}
We identify elements $w \in RF_{n-1}$ with their image by $t_n: RF_{n-1} \rightarrow \mathcal K_n$, that is, we denote by $w$ the automorphism fixing all $x_i$ save $x_n$, which is sent to $x_n^w$. Applying Lemma \ref{generating_the_kernel} to the semi-direct product decomposition \eqref{dec2}, we see that $\mathcal K_n'$ is generated by the elements $\chi_{in}^w$, which we now compute. The automorphism $\chi_{in}^w$ fixes $x_\alpha$ if $\alpha \notin \{i,n\}$. On $x_i$ and $x_n$, using that $\chi_{in}(w) \equiv w \pmod{x_n}$, we compute:
\[\chi_{in}^w:
\left\{\begin{array}{clclclc}
x_i &\longmapsto &x_i        &\longmapsto &x_i^{x_n}                            &\longmapsto &x_i^{x_n^w}, \\
x_n &\longmapsto &{}^w\! x_n &\longmapsto &{}^{\chi_{in}(w)}\! x_n = {}^w\! x_n &\longmapsto &x_n.
\end{array}  \right.\]
From this calculation, we see that all $\chi = \chi_{in}^w$ commutes with every $\chi' = \chi_{jn}^v$, showing that $\mathcal K_n'$ is indeed abelian. If $j \neq i$, this is a consequence of the fact that these automorphisms act trivially modulo $x_n$:
\[\chi'(x_i^{x_n^w}) = x_i^{x_n^{\chi'(w)}} = x_i^{x_n^{w}}.\]
For $i=j$, it follows from the fact that the conjugates of $x_n$ commute.

Consider now $N_i$ the subgroup generated by the $\chi_{in}^w$, for $w \in RF_{n-1}$. All the elements of $N_i$ are automorphisms fixing all $x_j$ save $x_i$, and sending $x_i$ to an element $x_i^u$, for some $u \in \mathcal N(x_n)$. As a consequence, the map $c_i$ is a surjection from $\mathcal N(x_n)$ onto $N_i$. Since, by definition of the reduced free group, $x_i^{sx_it} = x_i^{st}$ for all $s,t \in RF_n$, we see that $c_i(v)$ depends only on the class $\overline v$ of $v$ in $RF_{n-1}/x_i$. We use this to show that $c_i$ is a morphism:
\[c_i(u)c_i(v): x_i \mapsto c_i(u)(x_i^{\overline v}) = (x_i^{\overline u})^{c_i(u)(\overline v)} = x_i^{\overline{uv}} = c_i(uv)(x_i). \]
Now, the kernel of $c_i$ is $C(x_i) \cap \mathcal N(x_n) = \mathcal N(x_i) \cap \mathcal N(x_n)$ (using Lemma \ref{centralizers_in_RF}). It thus induces an isomorphism between $\mathcal N(x_n)/(\mathcal N(x_i) \cap \mathcal N(x_n))$ and $N_i$. Moreover, since it is the image of $\mathcal N(x_n)$ in $RF_n/\mathcal N(x_i)$, this group identifies with the normal closure of $x_n$ inside $RF_n/x_i \cong RF_{n-1}$. 

We are left to show that $c_i$ is $RF_{n-1}$-equivariant. It is enough to show that it commutes with the actions of the generators. 
If $\varphi \in \langle \chi_{nj} \rangle_{j \neq i}$, then $x_i$ does not appear in $\varphi(x_n)$, so that:
\[c_i(u)^{\varphi}:
\left\{\begin{array}{clclclc}
x_i &\longmapsto &x_i &\longmapsto &x_i^u  &\longmapsto &x_i^{\varphi(u)}, \\
x_n &\longmapsto &\varphi(x_n) &\longmapsto &\varphi(x_n) &\longmapsto &x_n,
\end{array}  \right.\]
showing that $c_i(u)^{\varphi} = c_i(\varphi(u))$. It remains to check that $c_i(u)^{\chi_{ni}} = c_i(\chi_{ni}(u))$. $c_i(\chi_{ni}(u))$ identifies with $c_i(u)$, since $\chi_{ni}$ acts trivially modulo $x_i$. We thus need to check that $\chi_{ni}$ commute with all $c_i(u)$ (which are all elements in $N_i$). This comes from the two relations $x_n^{x_i} = x_n^{x_i^u}$ (because $u \in \mathcal N(x_n)$) and $x_i^{\chi_{ni}(u)} = x_i^u$ (because $\chi_{ni}$ acts trivially modulo $x_i$).
\end{proof}

\subsection{The Lie algebra of the reduced McCool group}

The decomposition of $hP\Sigma_n$ described in Theorem \ref{dec_of_hMcCool} induces a decomposition of its Lie algebra:

\begin{theo}\label{L(wPn)}
The Lie algebra $\Lie(hP\Sigma_n)$ decomposes into a semi-direct product:
\[\Lie(hP\Sigma_n) \cong \left[\left( \prod\limits_{i <n} \langle y_i \rangle\right) \rtimes R \mathfrak L_{n-1}\right] \rtimes \Lie(hP\Sigma_{n-1}),\]
where $ \langle y_i \rangle$ is the ideal generated by $y_i$ inside $R \mathfrak L_{n-1}$, and the action of $R \mathfrak L_{n-1}$ on the product is the diagonal one.
\end{theo}

\begin{proof}
From the almost-direct product decomposition $hP\Sigma_n \cong \mathcal K_n \rtimes hP\Sigma_{n-1}$, comes a decomposition of the Lie algebra $\Lie(hP\Sigma_n) \cong \Lie(\mathcal K_n) \rtimes \Lie(hP\Sigma_{n-1})$. In the decomposition of $\mathcal K_n$ described above \eqref{dec2}, we can replace the normal closure $\mathcal N(x_n)/x_i$ of $x_n$ in $RF_n/x_i$ by the normal closure $\mathcal N(x_i)/x_n$ of $x_i$ in $RF_n/x_n \cong RF_{n-1}$. Indeed, the automorphism of $RF_n$ exchanging $x_i$ and $x_n$ induces an isomorphism between these two, which is $RF_{n-1}$-equivariant, since $x_i$ acts trivially on both of them. We thus have to compute:
\[\Lie(\mathcal K_n) \cong \Lie\left[\left( \prod\limits_{i <n} \mathcal N(x_i)\right) \rtimes RF_{n-1}\right].\]
Since this is not a decomposition into an almost direct product, we have to use \cite[\textsection 3.1]{Darne2}: we need to compute $\Gamma_*^{RF_{n-1}}\left( \prod\mathcal N(x_i)\right)$, which is the product $\prod \Gamma_*^{RF_{n-1}}(\mathcal N(x_i))$, since $RF_{n-1}$ acts diagonally. In order to do this, consider the split short exact sequence of groups: 
\[\mathcal N(x_i) \hookrightarrow RF_{n-1} \twoheadrightarrow RF_{n-1}/x_i \cong RF_{n-2}.\]
From \cite[Prop.~3.4]{Darne2}, this gives rise to a decomposition of $\Gamma_*(RF_{n-1})$ into a semi-direct product $\Gamma_*^{RF_{n-2}}(\mathcal N(x_i)) \rtimes \Gamma_*(RF_{n-2})$, where $\Gamma_*^{RF_{n-2}}(\mathcal N(x_i))$ is defined by taking commutators with $\mathcal N(x_i) \rtimes RF_{n-2} \cong RF_{n-1}$ at each step, so is equal to $\Gamma_*^{RF_{n-1}}(\mathcal N(x_i))$. As a consequence, $\mathcal N_*(x_i):= \Gamma_*^{RF_{n-1}}(\mathcal N(x_i))$ is the intersection of $\Gamma_*(RF_{n-1})$ with $\mathcal N(x_i)$. Its associated Lie algebra fits into the short exact sequence:
\[\Lie(\mathcal N_*(x_i)) \hookrightarrow \Lie(RF_{n-1}) \twoheadrightarrow \Lie(RF_{n-2}).\]
Theorem \ref{L(RF) = LR} ensures that the projection on the right identifies with the projection of $R\mathfrak L_{n-1}$ onto $R\mathfrak L_{n-2}$ sending $y_i$ to $0$, whose kernel is $\langle y_i \rangle$. Thus $\Lie(\mathcal N_*(x_i)) \cong \langle y_i \rangle$, and $\Lie(\mathcal N(x_i) \rtimes RF_{n-1}) \cong  \Lie(\mathcal N_*(x_i)) \rtimes \Lie(RF_{n-1}) \cong \langle y_i \rangle \rtimes R\mathfrak L_{n-1}$, ending the proof of the theorem.
\end{proof}

\subsection{The Andreadakis equality}\label{par_Andreadakis}

Theorem \ref{L(wPn)} gives a complete description of the graded Lie ring associated to $\Gamma_*(hP\Sigma_n)$. On the other hand, Theorem \ref{Johnson_conj_iso} describes the Lie ring associated with the Andreadakis filtration $hP\Sigma_n \cap \mathcal A_*(RF_n)$. Using these two results, we are now able to show:

\begin{theo}\label{Andreadakis_for_wP_n}
The Andreadakis equality holds for $hP\Sigma_n$.
\end{theo}

\begin{proof}
We want to show that the Johnson morphism $\tau': \Lie(hP\Sigma_n) \rightarrow \Der(R \mathfrak L_n)$ is injective (see the end of \cref{intro_to_Andreadakis}). We make use of the following the commutative diagrams:
\[\begin{tikzcd}
\Lie(\mathcal K_n) \ar[r, hook]      \ar[d, dashed, "\tau'"] 
&\Lie(hP\Sigma_n)  \ar[r, two heads] \ar[d, "\tau'"]
&\Lie(hP\Sigma_{n-1})                    \ar[d, "\tau'"] \\
\bullet                  \ar[r, hook]     
&\Der_\tau(R\mathfrak L_n)   \ar[r, two heads]
&\Der_\tau(R\mathfrak L_{n-1}),                    
\end{tikzcd}\]
Where the bottom projection is the one induced by $y_n \mapsto 0$.
By induction (beginning at $n = 2$), using the snake lemma, we only have to prove that the left map is injective, that is, that $\tau': \Lie(\mathcal K_n) \rightarrow \Der(R\mathfrak L_n)$ is.
Take an element 
\[\varphi = ((w_i), w_n)  \in \Gamma_j(\mathcal K_n) = \left(\prod\limits_{j < n} \left( \Gamma_j(RF_{n}) \cap \mathcal N(x_n)\right)/x_i\right) \rtimes \Gamma_j(RF_{n-1}),\]
meaning that $\varphi$ is the automorphism conjugating $x_n$ by $w_n \in \Gamma_j(RF_{n-1})$ and $x_i$ by $w_i \in \Gamma_j(RF_n) \cap \mathcal N(x_n)$ for $i < n$, which depends only on the class of each $w_i$ modulo $\mathcal N(x_i)$. Then $\tau'_j( \overline \varphi)$ sends each $y_i$ ($i \leq n$) to $[y_i, \overline w_i] \in \mathcal L_{j+1}(RF_n)$. As a consequence, the equality $\tau'_j( \overline \varphi) = 0$ would mean that each $\overline w_i$ commutes with $y_i$ in $\mathcal L(RF_n) \cong R \mathcal L_n$. However, by Lemma \ref{centralizers_in_RL}, this would imply that $\overline w_i \in \langle y_i \rangle$. However, in the course of the proof of Theorem \ref{L(wPn)}, we have shown that  $\langle y_i \rangle = \mathcal L\left(\Gamma_*(RF_n) \cap \mathcal N(x_i)\right)$. Thus there exists $v_i$ in $\Gamma_j(RF_n) \cap \mathcal N(x_i)$ such that $\overline v_i = \overline w_i$, that is, $w_i \equiv v_i \pmod{\Gamma_{j+1}(RF_n)}$. But we can replace $w_i$ by $w_i v_i ^{-1}$ without changing $\varphi$, so that all the $w_i$ can be chosen to be in $\Gamma_{j+1}(RF_n)$. This implies that $\varphi \in \Gamma_{j+1}(hP\Sigma_n)$, which means that $\overline \varphi = 0$ in $\Lie_j(hP\Sigma_n)$. This ends the proof that the kernel of $\tau'$ is trivial, and the proof of the Theorem.
\end{proof}

\subsection{Braids up to homotopy}\label{par_hPn}

Consider the (classical) pure braid group $P_n$. It can be embedded into the monoid of string-links on $n$ strands. These string-links can be considered \emph{up to (link-)homotopy}, which means that one adds to the isotopy relation the possibility for each strand to cross itself. This relation is obviously compatible with the monoid structure, and since every string-link is in fact homotopic to a braid, this quotient is a quotient of the pure braid group, called the group of braids up to homotopy, denoted by $hP_n$.

\subsubsection{Decomposition and Lie algebra}\label{dec_of_hPn}

In \cite{Goldsmith-Braids}, Goldsmith described $hP_n$ as a quotient of $P_n$ by a finite set of relations. These relations say exactly that for $j < k$, the generators $A_{jk}$ commute with their conjugates by elements of $\langle A_{ik}\rangle_{i < k} \cong F_{k-1}$. This means exactly that the free factors in the decomposition of $P_n$ are replaced by reduced free groups:
\[hP_{n+1} \cong RF_n \rtimes hP_n.\]
This decomposition first appeared explicitly in \cite{Habegger-Lin}, where a more topological proof is described.

Such a decomposition is compatible with the decomposition of the (classical) pure braid group, which means that the canonical projections give a morphism of semi-direct products:
\begin{equation}\label{proj_between_dec}
\begin{tikzcd}
F_n  \ar[r, hook] \ar[d, two heads] &P_{n+1}  \ar[r, two heads] \ar[d, two heads] &P_n   \ar[l, bend right] \ar[d, two heads]\\ 
RF_n \ar[r, hook]                   &hP_{n+1} \ar[r, two heads]                   &hP_n. \ar[l, bend right]
\end{tikzcd}
\end{equation}

Since Goldsmith's relations are commutation relations, the projection from $P_{n+1}$ onto $hP_{n+1}$ induces an isomorphism between $P_{n+1}^{ab}$ onto $hP_{n+1}^{ab}$. As a consequence, since the decomposition $P_{n+1} \cong F_n \rtimes P_n$ is an almost-direct product decomposition, the decomposition $hP_{n+1} \cong RF_n \rtimes hP_n$ also is. It thus induces a decomposition of the lower central series and of the corresponding Lie ring. Precisely, we get iterated semi-direct product decompositions:
\begin{equation}\label{dec_LCS_hPn}
\Gamma_j(hP_{n+1}) =  \Gamma_j(RF_n) \rtimes \Gamma_j(hP_n),
\end{equation}
which induces such decompositions on the associated graded Lie rings. Thus we get:
\begin{prop}\label{dec_L(hPn)}
The group $hP_{n+1}$ is $n$-nilpotent, and its Lie algebra decomposes as an iterated semi-direct product of reduced free Lie algebras:
\[\Lie(hP_{n+1}) \cong \Lie(RF_n) \rtimes \Lie(hP_n) \cong R\mathfrak L_n \rtimes \Lie(hP_n).\]
\end{prop}

From this, we can deduce the Hirsch rank of $hP_n$, recovering Milnor's formula, as quoted in \cite[Section 3]{Habegger-Lin}:
\begin{cor}
The group $hP_n$ has no torsion and its Hirsch rank is: 
\[\rk(hP_n) = \sum_{k = 1}^{n-1} (k-1)! \binom{n}{k+1}.\]
\end{cor}

\begin{proof}
That it has no torsion (even no torsion in its lower central series) comes from the fact that the $R\mathfrak L[m]$ do not, according to Proposition \ref{rk_of_RL}. The same proposition gives us the ranks of the $R\mathfrak L[m]_k$, allowing us to compute:
\[\rk\left( \Lie_k(hP_n) \right) = \sum\limits_{m = 1}^{n-1} \rk\left(R\mathfrak L[m]_k\right) = (k-1)! \sum\limits_{m = 1}^{n-1}\binom{m}{k} = (k-1)! \binom{n}{k+1},\]
the last equality being obtained by iterating Pascal's formula, or by a combinatorial proof (replacing the choice of $k$ elements $t_1, ..., t_k$ among $m$ elements, with $m$ ranging from $k$ to $n-1$, by the choice of $k+1$ elements $t_1, ..., t_k, m+1$ among $n$ elements).
\end{proof}

Let us also mention that we can deduce from the decomposition of $\Lie(hP_n)$ described in Proposition \ref{dec_L(hPn)} and from the usual presentation of the pure braid group a presentation of this Lie ring, which is a quotient of the Drinfeld-Kohno Lie ring $\Lie(P_n)$ of infinitesimal braids (whose rational version was introduced in \cite{Kohno}).

\begin{cor}\label{hDrinfeld-Kohno}
The Lie ring of $hP_n$ is generated by $t_{ij}\ (1 \leq i , j \leq n)$, under the Drinfeld-Kohno relations:
\[\begin{cases}
t_{ij} = t_{ji},\ t_{ii} = 0 &\forall i,j,\\ 
[t_{ij}, t_{ik} + t_{kj}] = 0 &\forall i, j, k,\\
[t_{ij}, t_{kl}] = 0 &\text{if } \{i,j\} \cap \{k,l\} = \varnothing,
\end{cases}\]
to which are added, for each $m$, the vanishing of Lie monomials in the $t_{im}$ ($i<m$) with repetition.
\end{cor}

\begin{proof}
The proof in the classical case (see for instance to the appendix of \cite{Darne2}) adapts verbatim, by considering reduced free Lie rings instead of free Lie rings.
\end{proof}

Remark that as in the definition of the reduced free Lie ring (Definition \ref{def_of_RL} -- see also Remark \ref{finite_pstation_RL}), one can give a simpler finite presentation by considering, for each $m$, only linear Lie monomials in the $t_{im}$ ($i<m$) of length at most $m$.

\subsubsection{The Andreadakis problem}\label{par_Andreadakis_hPn}

The semi-direct product $RF_n \rtimes hP_n$ described above is the same thing as an action of $hP_n$ on $RF_n$, also described by a morphism from $hP_n$ to $\Aut(RF_n)$. This is the \emph{homotopy Artin action}, that we now study, using the fact that it is encoded by conjugation inside $P_{n+1} = RF_n \rtimes hP_n$.

First, remark that this action is by basis-conjugating automorphisms. In fact, the compatibility diagram \eqref{proj_between_dec} gives rise to a commutative diagram:
\[\begin{tikzcd}
P_n  \ar[r, hook] \ar[d, two heads] &\Aut_C(F_n) \ar[d, two heads]\\ 
hP_n \ar[r]                         &\Aut_C(RF_n),
\end{tikzcd}\]
the morphism on the left being surjective by \ref{generators_of_hMcCool}. The top map, which is the Artin action, is injective (the action is faithful) and its image is exactly the subgroup of basis-conjugating automorphisms fixing the \emph{boundary element} $x_1 \cdots x_n$ \cite[Th.\ 1.9]{Birman}. Habegger and Lin have shown that the analogous statements are true for $hP_n$ \cite[Th.\ 1.7]{Habegger-Lin}: the homotopy Artin action induces an isomorphism between $hP_n$ and the group $\Aut_C^\partial(RF_n)$ of basis-conjugating automorphisms of $RF_n$ preserving the product $x_1 \cdots x_n$. In their proof, they show that this group admits the same decomposition as $hP_n$, and the pieces of these decompositions identify under the Artin morphism. We will recover the faithfulness of the homotopy Artin action as part of our answer to the Andreadakis problem for $hP_n \subset \Aut_C(RF_n)$ (see Cor.\ \ref{Artin_faithful} below).

\begin{theo}\label{Andreadakis_for_hP_n}
The Andreadakis equality holds for $hP_n$, embedded into $\Aut(RF_n)$ \emph{via} the Artin action.
\end{theo}

\begin{proof}
We adapt the proof for $P_n$ given in \cite{Darne2}. Let $w \in hP_n$, and suppose that $w$ acts on $RF_n$ as an element of $\mathcal A_j$. We want to show that it belongs to $\Gamma_j(hP_n)$. Our hypothesis can be written as:
\[[w, RF_n] \subseteq \Gamma_{j+1}(RF_n),\]
where the bracket is computed in $RF_n \rtimes hP_n$, which is exactly $hP_{n+1}$. Moreover, from the decomposition of the lower central series of $hP_{n+1}$ described above (\cref{dec_of_hPn}), we deduce that $\Gamma_j(hP_n) = hP_n \cap \Gamma_j(hP_{n+1})$, so that the conclusion we seek is in fact $w \in \Gamma_j(hP_{n+1})$. Let us comb $w$: we write $w = \beta_n \cdots \beta_2 \in RF_{n-1} \rtimes (RF_{n-2}  \rtimes( \cdots \rtimes RF_1) = hP_n$. Again, because of the decomposition of the lower central series of $hP_n$, we need to show that each $\beta_i$ is in $\Gamma_j(P_{n+1})$. In the rest of the proof, we often write $\Gamma_k$ for $\Gamma_k(hP_{n+1})$, its intersection with the subgroups we consider being their own lower central series, because of \eqref{dec_LCS_hPn}.

Let us suppose that $w \notin \Gamma_j(hP_{n+1})$. Then $w \in \Gamma_k - \Gamma_{k+1}$ for some $k < j$. Let $i$ be maximal such that $\beta_i \notin \Gamma_{k+1}$. 
On the one hand, the generator $A_{i, n+1} \in RF_n$ commutes with all $\beta_k$ with $k < i$, so that $[w, A_{i, n+1}] \equiv [\beta_i, A_{i, n+1}] \pmod{\Gamma_{k+2}}$. Moreover, by hypothesis, $[w, A_{i, n+1}] \in \Gamma_{j+1} \subseteq \Gamma_{k+2}$, so that $[\beta_i, A_{i, n+1}] \in \Gamma_{k+2}$. Since $\beta_i$ has degree $k$ and $A_{i, n+1}$ has degree $1$ in the lower central series, this means that $[\overline{\beta_i}, \overline A_{i, n+1}] = 0$ in the Lie algebra. 
On the other hand, $\beta_i$ and $A_{i, n+1}$ belong to another copy of $RF_n$ inside $hP_{n+1}$, namely $\langle A_{1, i}, ..., A_{i-1,i}, A_{i, i+1}, ... A_{i, n+1} \rangle$. We denote this copy by $\widetilde{RF_n}$. Remark that the equality $\Gamma_*(\widetilde{RF_n}) = \widetilde{RF_n} \cap \Gamma_*(hP_{n+1})$ is also true for this copy of $RF_n$, as one sees by switching the strands $i$ and $n+1$ in the reasoning above. But then we can apply Lemma \ref{centralizers_in_RL}: since $\overline{\beta_i}$ commutes with the generator $\overline A_{i, n+1}$ of $\Lie(\widetilde{RF_n}) \cong R\mathfrak L_n$, it must belong to the Lie ideal of $\Lie(\widetilde{RF_n})$ generated by $\overline A_{i, n+1}$. But this is impossible: by definition of $\beta_i$, the generator $\overline A_{i, n+1}$ cannot appear in $\overline{\beta_i}$. We thus get a contradiction, and our conclusion.
\end{proof}
 
From this, we can recover the injectivity part of the result of Habegger and Lin:
 
\begin{cor}\label{Artin_faithful} \textup{\cite[Th.\ 1.7]{Habegger-Lin}.}
The homotopy Artin action is faithful. 
\end{cor}

\begin{proof}
If $w \in hP_n$ acts trivially on $RF_n$, then $w \in \{1\} = \mathcal A_n(RF_n)$, so that $w \in \Gamma_n(hP_n) = 1$.
\end{proof}

This injectivity of $hP_n \rightarrow hP\Sigma_n$ is weaker than our statement, which says that the lower central series are compatible, since they both are the trace of the Andreadakis filtration $\mathcal A_*(RF_n)$:

\begin{cor}
For all $n$, $hP_n \cap \Gamma_*(hP\Sigma_n) = \Gamma_*(hP_n)$.
\end{cor}

\begin{proof}
Combine Theorems \ref{Andreadakis_for_hP_n} and \ref{Andreadakis_for_wP_n}.
\end{proof}

\section{Topological interpretation}\label{Section4}

Consider the group $P_n$ of pure braids. \emph{Via} the decomposition $P_{n+1} \cong F_n \rtimes P_n$, we get and action of $P_n$ on the free group $F_n$, which is the classical \emph{Artin action}. Geometrically, it is best understood as the action of $P_n$, which is the \emph{motion group} of $n$ points in a plane, on the fundamental group of the plane with $n$ points removed. As mentioned above (\cref{par_Andreadakis_hPn}), this action is faithful, giving an embedding of $P_n$ into $\Aut(F_n)$, whose image is exactly the subgroup $\Aut_C^\partial(F_n)$ of automorphisms fixing the conjugacy class of each generator $x_i$, and preserving the \emph{boundary element} $x_1 \cdots x_n$ \cite[Th.\ 1.9]{Birman}.

An analogous statement is true for the group $P\Sigma_n$ of pure \emph{welded braids}. This group is a group of tube-shaped braids in $\mathbb R^4$, and can also be seen as the (pure) \emph{motion group} of $n$ unknotted circles in the three-dimensional space (see \cite{Damiani} on the different definitions on this group). It acts on the fundamental group of $\mathbb R^3$ with $n$  unknotted circles removed, which is again the free group $F_n$. This \emph{Artin action} is again faithful, and its image is exactly the subgroup $\Aut_C(F_n)$ of automorphisms fixing the conjugacy class of each generator $x_i$ \cite{Goldsmith-Motions}.

The same statements are true \emph{up to (link-)homotopy}. These have been recalled for braids in \cref{par_hPn}. For welded braids, link-homotopy of string links also makes sense (in the four-dimensional space), and for welded diagrams (which are another point of view on these objects), this relation correspond to virtualization of self-crossings. It has been shown in \cite[Th. 2.34]{ABMW} that the group of welded braids up to homotopy is isomorphic to the group $\Aut_C(RF_n) = hP\Sigma_n$ of automorphisms or $RF_n$ fixing the conjugacy class of each generator $x_i$.

We sum up the situation in the following table:
\begin{equation}
\begin{array}{c|c}
\textbf{Up to isotopy} &\textbf{Up to homotopy} \\[6pt]
\begin{tikzcd}
P_n \ar[d, hook] \ar[r, "\cong"] & \Aut_C^\partial(F_n)  \ar[d, hook]\\
P\Sigma_n        \ar[r, "\cong"] & \Aut_C(F_n). 
\end{tikzcd}
&\begin{tikzcd}
hP_n \ar[d, hook] \ar[r, "\cong"] & \Aut_C^\partial(RF_n)  \ar[d, hook]\\
hP\Sigma_n        \ar[r, "\cong"] & \Aut_C(RF_n). 
\end{tikzcd}
\end{array}
\end{equation}

\subsection{Milnor invariants}

Here we interpret our work in terms of \emph{Milnor invariants} of welded braids up to homotopy. Milnor invariants were first defined in \cite{Milnor-IL} for links, as integers with some indeterminacy. It appeared later that they were more naturally defined for string links, for which they are proper integers, the indeterminacy previously observed corresponding exactly to a choice of presentation of a link as the closure of a string-link. Here we focus on their definition for braids, which is not a restrictive choice when working up to homotopy. 

If $\beta$ is a pure braid, we can look at its image \emph{via} the Artin action, which is a basis-conjugating automorphism $x_i \mapsto x_i^{w_i}$. The element $w_i$ is well-defined up to left multiplication by $x_i^{\pm 1}$, so it is well-defined if we suppose that $x_i$ does not appear in the class $\overline w_i \in F_n^{ab}$. For each $i$, one can look at the image of the element $w_i \in F_n$ by the \emph{Magnus expansion} $\mu: F_n \hookrightarrow \widehat{T[n]}$, getting an element of the completion of the free associative ring $\widehat{T[n]}$ on $n$ generators $X_1, ..., X_n$, which can be seen as the ring of non-commutative power series on these generators. Recall that the Magnus expansion is defined by $x_i \mapsto 1+X_i$, and it is is an injection of the free group $F_n$ into $\widehat{T[n]}^\times$. Then the \emph{Milnor invariants} are the coefficients of the $\mu(w_i)$. Precisely, if $i \leq n$ is an integer, and $I = (i_1, ..., i_d)$ is any list of positive integer, then $\mu_{I,i}(\beta)$ is the coefficient of the monomial $X_{i_1} \cdots X_{i_d}$ in $\mu(w_i)$. Moreover, we call $d$ the \emph{degree} of the Milnor invariant $\mu_{I,i}$.

The first non-trivial Milnor invariants of $\beta$ can also be obtained through the Johnson morphism. Namely, let $d$ be the greatest integer such that $\beta \in \mathcal A_d(F_n)$ (we identify $\beta$ with its image \emph{via} the Artin action). By definition of $w_i$, $x_i$ does not appear in the class $\overline w_i \in F_n^{ab}$. Thus, we deduce from \cite[Lem. 6.3]{Darne2} that for all $j \geq 1$, $[x_i,w_i] \in \Gamma_{j+1}(F_n) \Leftrightarrow w_i \in \Gamma_d(F_n)$. This implies that $d$ is maximal such that all $w_i$ belong to $\Gamma_j(F_n)$. The image of $\overline \beta \in  \mathcal A_d/ \mathcal A_{d+1}$ by the Johnson morphism is the derivation of the free Lie algebra $\mathfrak L[n]$ given by $x_i \mapsto [x_i, \overline w_i]$, where $\overline w_i \in \Gamma_d/\Gamma_{d+1}(F_n) \cong \mathfrak L[n]_d$ is the class of $w_i$, possibly trivial (but non-trivial for at least one $i$). 

Now, we can watch the element $\overline w_i$ as being inside $T[n]_d$, and the inclusion of $\mathfrak L[n]$ into $T[n]$ is exactly the graded map induces by the Magnus expansion $\mu$. Precisely, if we call $\widehat T_1^d$ the ideal of $T[n]$ defined by elements of valuation at least $d$ (the valuation of a power series being the total degree of its least nontrivial monomial), then $\Gamma_d(F_n) = \mu^{-1}(1+\widehat T_1^d)$, and the induced map $\overline \mu: \Gamma_d/\Gamma_{d+1}(F_n) \hookrightarrow \widehat T_1^d/\widehat T_1^{d+1}$ identifies with the canonical inclusion of $\mathfrak L[n]_d$ into $T[n]_d$. As a consequence, the class $\overline w_i$ is the degree-$d$ part of $\mu(w_i)$, which has valuation at least $d$. We sum this up in the following:

\begin{prop}
The group $\mathcal A_d(F_n) \cap P_n$ is the set of braids with vanishing Milnor invariants of degree at most $d-1$. Moreover, Milnor invariants of degree $d$ of these braids can be recovered from their image by the Johnson morphism $\tau: \mathcal A_d/\mathcal A_{d+1} \hookrightarrow \Der_d(\mathfrak L[n])$.
\end{prop}

Obviously, since we have not used anywhere that the automorphism $\beta$ preserves the boundary element, these constructions work for all welded braids (that is, for all basis-conjugating automorphisms of $F_n$).

Let us now explain how to define Milnor invariants for (welded) braids up to homotopy. First, we need to replace $F_n$ by $RF_n$. Then we can assume that $x_i$ does not appear in $w_i$ (since $x_i^{ux_iv} = x_i^{uv}$ in the reduced free group). The Magnus expansion must be replaced by the morphism \eqref{canonical_inj}, and we get only Milnor invariants without repetitions (that is, $I$ must be without repetition in order to define a non-trivial $\mu_{I,i}$). Everything works as described above (using the work done in \cref{par_A[Y]}), so that $\mathcal A_d(F_n)$ is exactly the subgroup where invariants of degree at most $d-1$ vanish. So we can reformulate our theorems \ref{Andreadakis_for_wP_n} and \ref{Andreadakis_for_hP_n} as:

\begin{theo}\label{Andreadakis_Milnor}
Homotopy Milnor invariants of degree at most $d$ classify braids up to homotopy (resp.\ welded braids up to homotopy) up to elements of $\Gamma_{d+1}(hP_n)$ (resp. up to elements of $\Gamma_{d+1}(hP\Sigma_n)$).
\end{theo}

\begin{rmq}
The group $\Gamma_{d+1}(hP_n)$ can also be seen as the set of braids which are homotopic to elements of $\Gamma_{d+1}(P_n)$.
\end{rmq}

\subsection{Arrow calculus}

We now explain briefly the precise link between our work and the work of Meilhan and Yasuara in \cite{Meilhan-Yasuhara}. We will not give any definition here; the reader is referred to their paper for basic definitions and details. 

Our understanding of the link between our work and theirs relies on the following remark: \textbf{Ccalculus of arrows and $w$-trees is the same thing as commutator calculus in the welded braid group $P\Sigma_n$}. Precisely, when attaching a tree $T$ to a diagram $D$, one has to select the points where the root and leaves of $T$ are attached. If we consider a little arc around each of these points, we see that doing so consists in choosing $n$ strands (which inherit there orientation from $D$). Then the data of $T$ describes an element of the braid group on these strands, and doing the surgery along $T$ is exactly the same as inserting the braid described by $T$ at the chosen spot on $D$, to get the new diagram $D_T$. Namely, a single arrow from a strand $j$ to a strand $i$ describes the insertion of the braid $\chi_{ij}$, and a tree with root at $i$ describes the insertion of a commutator between the $\chi_{ij}$, for varying $j$ (note that any number of strands can be added). 

In the light of this remark, we can see that many relations they describe correspond to algebraic relations written in the present paper. Also, two diagrams are $w_k$-equivalent if and only if they can be obtained from one another by inserting braids in $\Gamma_k(P_n)$ (for varying $n$). And we can in fact deduce our Andreadakis equality (Theorem \ref{Andreadakis_Milnor}) from their classification theorem of welded string links up to homotopy \cite[Th.\ 9.4]{Meilhan-Yasuhara}. They fell short of doing so, stating only their weaker corollary 9.5. In fact, they did not look for the precise identification between trees and commutator calculus that I have described here. They only knew that something of the sort should be true, but were interested in other matters at the time.

\section{A presentation of the homotopy loop braid group}\label{par_pstation}

In \cite{Goldsmith-Braids}, Goldsmith gave a presentation of the braid group up to homotopy (see also \cref{par_hPn}). She proved that, to a presentation of the pure braid group with generators $A_{ij}$, one has to add the family of relations making each $\langle A_{1k}, ..., A_{(k-1)k} \rangle$ into a reduced free group. The goal of the present section is to  give a similar presentation of the loop braid group up to homotopy. The situation here it more intricate: to a presentation of $P\Sigma_n$ with generators $\chi_{ij}$, we have to add three families of relations:
\begin{relations}
  \item The relations saying that for all $m$, $\langle \chi_{mk} \rangle_{k < m}$ is reduced.
  \item $[\chi_{im}, w, \chi_{jm}] = 1$, for $i,j < m$ and $w \in \langle \chi_{mk} \rangle_{k < m}$.
  \item $[\chi_{im}, w, \chi_{mi}] = 1$, for $i < m$ and $w \in \langle \chi_{mk} \rangle_{k < m, k \neq i}$.
\end{relations}

Remark that because of the symmetry with respect to the generators of $RF_n$, these relations are still true if we replace each symbol $"<"$ by a symbol $"\neq"$, which would give a more symmetric set of relations.

\begin{rmq}
These relations also describe the quotient of the group $wB_n$ of all welded braids by the homotopy relation. Indeed, performing a homotopy cannot move endpoints of string links, so that the subgroup of relations must be a subgroup of the \emph{pure} welded braid group, like in the classical setting \cite[Lem. 1]{Goldsmith-Braids}.
\end{rmq}

\subsection{Generators of nilpotent groups}

One key argument in the determination of a presentation of $hP\Sigma_n$ consists in lifting generators from Lie rings to groups. Such generators will be obtained from combinatorics in the free Lie ring (see our appendix), and their lifting will use the nilpotence of the groups involved.

\begin{conv}
By a \emph{finite} filtration, we always mean a separating one: a strongly central series $G_*$ is \emph{finite} if their exists a $i \geq 1$ such that $G_i =\{1\}$. In particular, if there exists a finite strongly central series on $G$, then $G$ must be nilpotent (recall that $G_i \supseteq \Gamma_iG$).
\end{conv}

\begin{lem}\label{gen_of_nilp}
Let $G_*$ be a finite strongly central filtration on a (nilpotent) group $G$. Suppose that the $x_\alpha$ are elements of $G$ such that their classes $\overline x_\alpha$ generate the Lie ring $\Lie(G_*)$. Then the $x_\alpha$ generate $G$.
\end{lem}

\begin{proof}
Consider the subgroup $K$ of $G$ generated by the $x_\alpha$. The canonical morphism from $\Lie(G_*\cap K)$ to $\Lie(G_*)$ comes from an injection between filtrations, so it is injective. By hypothesis, it is also surjective. By induction (using the five-lemma), we deduce that $K/(G_j \cap K) = G/G_j$, for all $j$. Since there exists $j$ such that $G_j = \{1\}$, this proves that $K=G$, whence the conclusion.
\end{proof}

The definition of the Lyndon monomials $P_w$ (\cref{par_Lyndon_basis}) makes sense in any group, if we interpret letters as elements of the group, and brackets as commutators. 

\begin{prop}\label{reduced_normal_closure}
Let $G$ be a nilpotent group generated by a set $X$, and $x \in X$. Then the normal closure $\mathcal N(x)$ of $x$ in $G$ is generated by Lyndon monomials $P_w$, for Lyndon words $w \in X^*$ containing $x$.
\end{prop}

\begin{proof}
By taking images in $G$, it is enough to show this for the free nilpotent group $F_j[X]:= F[X]/\Gamma_{j+1}$. In this case, $\mathcal N(x)$ is the kernel of the canonical projection from from $F_j[X]$ to $F_j[X-\{x\}]]$. Setting $\mathcal N_*(x):= \mathcal N(x) \cap \Gamma_*(F_j[X])$, we get a short exact sequence of filtrations translating into a short exact sequence of Lie rings:
\[\begin{tikzcd} \Lie(\mathcal N_*(x)) \ar[r, hook] &\Lie(F_j[X]) \ar[r, two heads] & \Lie(F_j[X-\{x\}]). \end{tikzcd}\]
Since $\Lie(F_j[X])$ is the $j$-th truncation of the free Lie algebra on $Y = \overline X$, and the projection $p$ is the canonical one (sending $y = \overline x$ to $0$), the subring $\Lie(\mathcal N_*(x))$ identifies with the $j$-th truncation of the ideal $\langle y \rangle$ of $\mathfrak L [Y]$. This ideal is linearly generated by Lyndon Lie monomials on $Y$ containing $y$. Since these are the classes of the corresponding monomials in the group $F_j[X]$, Lemma \ref{gen_of_nilp} gives the desired conclusion.
\end{proof}

\begin{cor}\label{Basis_of_N(x_i)}
Let $X$ be a set, and $x \in X$. The normal closure $\mathcal N(x)$ of $x$ in $RF[X]$ is free abelian on the Lyndon monomials $P_w$, for Lyndon words without repetition $w \in X^*$ containing $x$.
\end{cor}

\begin{proof}
It is enough to show this for $X$ finite. Then $RF[X]$ is nilpotent, and we can apply Proposition \ref{reduced_normal_closure} to show that Lyndon monomials without repetition containing $x$ generate $\mathcal N(x)$. Indeed, in $RF[X]$, the only non-trivial Lyndon monomials in elements of $X$ are those without repetition. Moreover, $\mathcal N(x)$ is abelian, by definition of $RF[X]$. We are thus left with proving that these elements are linearly independent. But any non-trivial linear relation between them would give a non-trivial linear relation between Lyndon polynomials without repetition in $\Lie(RF[X])$ (take $l$ to be the minimal length of the monomials involves, and project the relation into $\Gamma_l/\Gamma_{l+1}$). Such a relation cannot hold (Proposition \ref{rk_of_RL}), so this proves the corollary.
\end{proof}

If $g, g_1, ..., g_m$ are elements of a group, let us denote by $\Lynd(g; g_1,...,g_m)$ the family of Lyndon monomials $(P_w)$, where $w$ runs through Lyndon words without repetition on the letters $g, g_1, ..., g_m$ which contain $g$. 
When considering these sets, we will choose an order on the letters making all $g_i$ greater than $g$. In that case, elements of $\Lynd(g; g_1,...,g_m)$ are of the form $[[g,P_v],P_w]$, where neither $v$ nor $w$ contains $g$. As usual, we denote by $(g_1,..., \widehat g_i, ..., g_m)$ the $(m-1)$-tuple obtained from $(g_1,..., g_m)$ by removing the $i$-th component.

We now use Cor.\ \ref{Basis_of_N(x_i)} in order to get a basis of the group $\mathcal K_n'$ introduced in \cref{section_Andreadakis} from the decomposition obtained in Lemma \ref{decomposition_lemma_2}.

\begin{lem}\label{basis_of_Kn'}
A basis of the abelian group $\mathcal K_n'$ is given by:
\[\bigcup\limits_i \Lynd(\chi_{in}; \chi_{n1},..., \widehat{\chi_{ni}}, ...,\chi_{n,n-1}).\]
\end{lem}

\begin{proof}
We use notations from the proof of Lemma \ref{decomposition_lemma_2}. Equivariance of the isomorphism $c_i$ ensures that $c_i^{-1}$ sends the set $\Lynd(\chi_{in}; \chi_{n1},..., \widehat{\chi_{ni}}, ..., \chi_{n,n-1})$ to the set $\mathcal B:= \Lynd(x_n; \chi_{n1},..., \widehat{\chi_{ni}}, ..., \chi_{n,n-1})$, the latter brackets being computed in the semi-direct product $\mathcal (N(x_n)/x_i) \rtimes \langle \chi_{nj} \rangle_j$. If $v \in RF_{n-1}$, we denote by $\chi_v$ the automorphism of $RF_n$ sending $x_n$ to $x_n^v$ and fixing all other generators  ($\chi_v$ was denoted by $t_n(v)$ above). Elements of $\mathcal B$ are of the form $[[x_n,\chi_v], \chi_w]$, where $\chi_v$ and $\chi_w$ are Lyndon monomials in the $\chi_{nj}$ ($j \neq i$), which means exactly that $v$ and $w$ are Lyndon monomials in the $x_j$ ($j \neq i,n$), since $t_n: v \mapsto \chi_v$ is a morphism. Recall that the class of $\chi_v$ in the Lie algebra $\Lie(\mathcal A_*(RF_n))$ acts on the Lie algebra $R\mathfrak L_n$ \emph{via} the tangential derivation $\tau(\overline \chi_v)$ induced by $[\chi_v,-]$, sending $x_n$ to $[x_n, v]$ and all other $x_i$ to $0$. As a consequence, the class of $[[x_n,\chi_v], \chi_w]$ in the Lie algebra $\Lie(\mathcal N_*(x_n)/x_i) \subset R\mathfrak L_n$ is:
\[\tau(\overline \chi_w)\tau(\overline \chi_v)(x_n) = \tau(\overline \chi_w)([v,x_n]) = [v,[w,x_n]] =[[x_n,w],v],\]
since the derivation $\tau(\overline \chi_w)$ vanishes on $v$. As a consequence, the family $\mathcal B$ is another lift of the basis of $\Lie(\mathcal N_*(x_n)/x_i)$ considered above, and the same proof as the proof of corollary \ref{Basis_of_N(x_i)} (in $RF_n/x_i \cong RF_{n-1}$) shows that it is a basis of $\mathcal N(x_n)/x_i$, whence the result.
\end{proof}

\begin{rmq}
In the semi-direct product $\mathcal (N(x_n)/x_i) \rtimes \langle \chi_{nj} \rangle_j$ which appears in the proof; the group $\langle \chi_{nj} \rangle_j$ is isomorphic to $RF_{n-1}$ but its action is \emph{not} the conjugation action.
\end{rmq}

\subsection{The presentation}

Let us recall the relations on the $\chi_{ij}$ that will give a presentation of $hP\Sigma_n$:
\begin{relations}
\setcounter{enumi}{-1}
  \item \label{item:R0} The McCool relations on the $\chi_{ij}$ (see the Introduction).
  \item \label{item:R1} $[\chi_{mi}, w, \chi_{mi}] = 1$, for $i < m$, and $w \in \langle \chi_{mk} \rangle_{k < m}$.
  \item \label{item:R2} $[\chi_{im}, w, \chi_{jm}] = 1$, for $i,j < m$ and $w \in \langle \chi_{mk} \rangle_{k < m}$.
  \item \label{item:R3} $[\chi_{im}, w, \chi_{mi}] = 1$, for $i < m$ and $w \in \langle \chi_{mk} \rangle_{k < m, k \neq i}$.
\end{relations}

We now show that they indeed give the presentation that we were looking for:

\begin{theo}\label{The_presentation}
The pure loop braid group up to homotopy $hP\Sigma_n$ is the quotient of $P\Sigma_n$ by relations \textbf{(\ref{item:R1})}, \textbf{(\ref{item:R2})} and \textbf{(\ref{item:R3})}. As a consequence, it admits the presentation:
\[hP\Sigma_n \cong \left\langle \chi_{ij}\ (i \neq j)\ \middle|\ \textbf{(\ref{item:R0})}, \textbf{(\ref{item:R1})}, \textbf{(\ref{item:R2})}, \textbf{(\ref{item:R3})}\right\rangle\]
\end{theo}

\begin{proof}
Let $\mathcal G_n$ be the group defined by the presentation of the theorem. The $\chi_{ij}$ in $hP\Sigma_n$ satisfy the above relations. As a consequence,  there is and obvious morphism $\pi$ from $\mathcal G_n$ to $hP\Sigma_n$. Since the $\chi_{ij}$ generate $hP\Sigma_n$ (Cor. \ref{generators_of_hMcCool}), this morphism is surjective. We need to show that it is an isomorphism. We will do that by showing that $\mathcal G_n$ admits a decomposition similar to that of $hP\Sigma_n$, and that the pieces in the two decompositions are isomorphic \emph{via} $\pi$. We do this in three steps, parallel to the proof of Theorem \ref{dec_of_hMcCool}.

\textbf{Step 1.} We define a projection $\tilde p_n$ from $\mathcal G_n$ to $\mathcal G_{n-1}$ by sending $\chi_{ij}$ to $\chi_{ij}$ if $n \notin \{i,j\}$, and $\chi_{in}$ and $\chi_{nj}$ to $1$. This morphism is well defined (from the presentations), and so is its obvious section $\tilde s_n: \mathcal G_{n-1} \hookrightarrow \mathcal G_n$. If we denote by $\tilde{\mathcal K}_n$ the kernel of $\tilde p_n$, we get a semi-direct product decomposition $\mathcal G_n = \tilde{\mathcal K}_n \rtimes \mathcal G_{n-1}$ that fits in the following diagram:
\[\begin{tikzcd}
\tilde{\mathcal K}_n \ar[r, hook] \ar[d, dashed] 
&\mathcal G_n \ar[r, two heads, swap, "\tilde p_n"] \ar[d, two heads, "\pi"] 
&\mathcal G_{n-1} \ar[d, two heads, "\pi"] \ar[l, bend right, swap, "\tilde s_n"]\\
\mathcal K_n \ar[r, hook]
&hP\Sigma_n \ar[r, two heads, swap, "p_n"]
&hP\Sigma_{n-1} \ar[l, bend right, swap, "s_n"]
  \end{tikzcd}\]
By induction (using the five-lemma), beginning with the isomorphism $\mathcal G_2 \cong hP\Sigma_2 \cong \Z^2$ (which is the group $\langle \chi_{12}, \chi_{21} \rangle$ of inner automorphisms of $RF_2$), we only need to show that the induced morphism between the kernels are isomorphisms.

\textbf{Step 2.} We can apply Lemma \ref{generating_the_kernel} to the above decomposition of $\mathcal G_n$; the proof of Lemma \ref{decomposition_lemma_1} only used the McCool relations, so it carries over without change to show that $\tilde{\mathcal K}_n$ is generated by the $\chi_{in}$ together with the $\chi_{nj}$. This shows directly that the map from $\tilde{\mathcal K}_n$ to $\mathcal K_n$ is surjective (this fact also comes from the snake lemma and the induction hypothesis). Consider the map $\tilde{\mathcal K}_n \rightarrow \mathcal K_n \twoheadrightarrow RF_n$, where the second map is the projection $q_n$ from $\mathcal K_n$ to $RF_{n-1}$ defined in the proof of Theorem \ref{dec_of_hMcCool}. This map sends the $\chi_{in}$ to $1$ and the $\chi_{nj}$ to the $x_j$. From the relations (\ref{item:R1}), we know that the assignment $x_j \mapsto \chi_{nj}$ defines a section $\tilde t_n$ from $RF_{n-1}$ to $\tilde{\mathcal K}_n$. This shows that the $\chi_{nj}$ generate a reduced free group inside $\tilde{\mathcal K}_n$. If we denote by  $\tilde{\mathcal K}_n'$ the kernel of $\tilde q_n = q_n \circ \pi$,  we get
a semi-direct product decomposition $\tilde{\mathcal K}_n = \tilde{\mathcal K}_n' \rtimes RF_{n-1}$, similar to \eqref{dec2}, that fits in the following diagram:
\[\begin{tikzcd}
\tilde{\mathcal K}_n' \ar[r, hook] \ar[d, dashed] 
&\tilde{\mathcal K}_n \ar[r, two heads, swap, "\tilde q_n"] \ar[d, two heads, "\pi"] 
&RF_{n-1} \ar[d, "\cong"] \ar[l, bend right, swap, "\tilde t_n"]\\
\mathcal K_n' \ar[r, hook]
&\mathcal K_n \ar[r, two heads, swap, "q_n"]
&RF_{n-1} \ar[l, bend right, swap, "t_n"]
  \end{tikzcd}\]

\textbf{Step 3.} In order to show that the induced projection $\pi: \tilde{\mathcal K}_n' \rightarrow \mathcal K_n'$ is an isomorphism, we need to investigate the structure of $\tilde{\mathcal K}_n'$. By Lemma \ref{generating_the_kernel}, it is generated by the $\chi_{in}^w$ for $w \in \langle \chi_{nj} \rangle \cong RF_{n-1}$, and the relations (\ref{item:R2}) say exactly that these commute with each other. Thus $\tilde{\mathcal K}_n’$ is abelian. 
Let us fix $i$ and denote by $\tilde N_i$ the subgroup generated by the $\chi_{in}^w$. It is the normal closure of $\chi_{in}$ in the subgroup $\tilde M_i$ generated by $\chi_{in}$ and the $\chi_{nj}$. Relations (\ref{item:R1}) and (\ref{item:R2}) imply that $\chi_{in}$ and the $\chi_{nj}$ commute with their conjugates in $M_i$, which is thus a quotient of $RF_n$. In particular, $\tilde M_i$ is nilpotent, and we can apply Proposition \ref{reduced_normal_closure} to get that $\tilde N_i$ is generated by Lyndon monomials in $\chi_{in}$ and the $\chi_{nj}$ containing $\chi_{in}$. We can even limit ourselves to the subset $\Lynd(\chi_{in}; (\chi_{nj})_j)$ of monomials without repetitions, the other ones being trivial by the argument above. Furthermore, the relations (\ref{item:R3}) say exactly that among these, the ones containing $\chi_{ni}$ vanish. Thus, the abelian group $\tilde N_i$ is generated by $\Lynd(\chi_{in}; \chi_{n1},..., \widehat{\chi_{ni}}, ...,\chi_{n,n-1})$. Because of Lemma \ref{basis_of_Kn'}, we know that these monomials are sent to linearly independant elements in $N_i$ (in fact, to a basis of this abelian group), so they must be a basis of $\tilde N_i$, and the projection $\pi$ induced a isomorphism between $\tilde N_i$ and $N_i$. The projection $\pi: \tilde{\mathcal K}_n' \rightarrow \mathcal K_n'$, being the direct product of these isomorphisms, is thus an isomorphism, which is the desired conclusion.
\end{proof}

\begin{rmq}
The same remarks made at the end of \cref{par_A[Y]} for $RF_n$ holds true for $hP\Sigma_n$: it is finitely generated and nilpotent (of class $n-1$), so it has a finite presentation. However, in order to write down such a finite presentation, we need a presentation of the free $(n-1)$-nilpotent group on $n^2$ generators $\chi_{ij}$. We can then add to such a presentation the relations relations similar to (\ref{item:R1}), (\ref{item:R2}) and (\ref{item:R3}) that are iterated brackets of the generators (of any shape) of length at most $n-1$ to get an explicit finite presentation of $hP\Sigma_n$. In other words, the latter relations give a finite presentation of $hP\Sigma_n$ \emph{as a $(n-1)$-nilpotent group}.
\end{rmq}

\clearpage

\subsection{A presentation of the associated Lie ring}\label{par_pstation_Lie}

Using the above methods, one can also find a presentation of the Lie ring associated to $hP\Sigma_n$, similar to the presentation of $\Lie(hP_n)$ given in Corollary \ref{hDrinfeld-Kohno}. 

\begin{prop}
The Lie ring of $hP\Sigma_n$ is generated by $x_{ij}\ (1 \leq i \neq j \leq n)$, under the relations:
\[\begin{cases}
[x_{ik} + x_{jk}, x_{ij}] = 0 &\text{for}\ i,j, k\ \text{pairwise distinct,}\\ 
[x_{ik}, x_{jk}] = 0 &\text{for}\ i,j, k\ \text{pairwise distinct,}\\
[x_{ij}, x_{kl}] = 0 &\text{if}\ \{i,j\} \cap \{k,l\} = \varnothing,
\end{cases}\]
to which are added, for each $m$, the following families of relations:
\[\begin{cases}
[x_{im}, [x_{mi}, t]] = 0, \\ 
[x_{im}, [x_{jm}, t]] = 0, \\
[x_{im}, [x_{mi},t]] = 0, 
\end{cases}\]
where, in each case, $t$ describes Lie monomials in the $x_{mk}\ (k<m)$.
\end{prop}

\begin{proof}
Since it is very similar to the proof of Theorem \ref{The_presentation}, we only outline the proof. Let $h \mathfrak p_n$ be the Lie ring defined by the presentation of the theorem. The relations are true for the classes of the $\chi_{ij}$ in $\Lie(hP\Sigma_n)$ (as direct consequences of the relations in the group $hP\Sigma_n$), so that $x_{ij} \mapsto \overline \chi_{ij}$ defines a projection $\pi$ from $h \mathfrak p_n$ onto $\Lie(hP\Sigma_n)$. One shows that $h \mathfrak p_n$ admits a decomposition similar to the decomposition of $\Lie(hP\Sigma_n)$ described in Theorem \ref{L(wPn)}. Indeed, the morphism from $h \mathfrak p_n$ to $h \mathfrak p_{n-1}$ sending $x_{ij}$ on $x_{ij}$ if $n \notin \{i,j\}$ and to $0$ else is a well-defined projection $p$, which is split. From the relations, reasoning as in the proof of Lemma \ref{decomposition_lemma_1}, one checks that the $x_{in}$ together with the $x_{ni}$ generate an ideal of $h \mathfrak p_n$, which has to be the kernel $\mathfrak k_n$ of $p$. They one argues exactly as in the proof of Theorem \ref{The_presentation} to show (using the first family of relations) that $\mathfrak k_n$ decomposes as a semi-direct product $\mathfrak k_n' \rtimes R \mathcal L_{n-1}$. Moreover, the projection $\pi$ is compatible with the decompositions of $h \mathfrak p_n$ and $\Lie(hP\Sigma_n)$. Using the five lemma, we see that we only have to check that $\pi$ induces and isomorphism between $\mathfrak k_n'$ and $\prod \langle y_i \rangle$. Since we know a basis of the target, whose elements are Lie monomials on the $\overline \chi_{in}$ and $\overline \chi_{ni}$, we are left with showing that the corresponding Lie monomials on the $x_{in}$ and $x_{ni}$ generate $\mathfrak k_n'$. Like in the proof of Theorem \ref{The_presentation}, the last two families of relations ensure exactly that, so that $\pi$ is indeed an isomorphism.
\end{proof}

\begin{rmq}
In the presentation, one can consider only the relations where $t$ is a linear monomial of length at most $m$.
\end{rmq}

\begin{rmq}
It is a difficult open question, very much related to the Andreadakis problem for $P\Sigma_n$, to decide whether the first three relations (the linearize McCool relations) define a presentation of the Lie ring of $P\Sigma_n$. It is only known to hold rationally \cite{Berceanu}. 
\end{rmq}

\clearpage

\begin{subappendices}

\section{Lyndon words and the free Lie algebra}

For the comfort of the reader, we gather here some basic facts about Lyndon words. These describe a basis of the free Lie algebra, and we give a self-contained proof of this classical result involving as little machinery as possible. Our main sources for this appendix were Serre's Lecture Notes \cite{Serre} and Reutenauer's book \cite[5.1]{Reutenauer}.

\subsection{Lyndon words}

Let $\mathcal A$ be a set (called an \emph{alphabet}) endowed with a fixed total order. We denote by $\mathcal A^*$ the free monoid generated by $\mathcal A$. Elements of $\mathcal A^*$ are \emph{words} in 
$\mathcal A$, that is, finite sequence of elements of $\mathcal A$. The set $\mathcal A^*$ is endowed with the usual dictionary order induced by the order on $\mathcal A$.

The length of a word $w$ is denoted by $|w|$. If $v$ and $w$ are words, $v$ is a \emph{suffix} (resp.\ a \emph{prefix}) of $w$ if there exists a word $u$ such that $w = uv$ (resp.\ $w = vu$). It is called \emph{proper} when it is nonempty and different from $w$.

\begin{defi}
The  \emph{standard factorisation} of a word $w$ of length at least $2$ is the factorisation $w = uv$ where $v$ is the smallest proper suffix of $w$.
\end{defi}

\begin{defi}
A \emph{Lyndon word} is a nonempty word that is minimal among its nonempty suffixes.
\end{defi}

\begin{lem}\label{lem_std_1}
If $w = uv$ is a standard factorisation, then $v$ is a Lyndon word, and if $w$ is Lyndon then so is $u$.
\end{lem}

\begin{proof}
The fact that $v$ is a Lyndon word is clear. Suppose that $w$ is a Lyndon word. Let $x$ be any proper suffix of $u$. Since $uv = w < xv$, if $x$ is not a prefix of $u$, then $u < x$. Otherwise, $u = xy$ for some nonempty $y$, but then  $xyv < xv$ implies $yv < v$, which contradicts the definition of $v$. 
\end{proof}

The following proposition is the most basic result in the theory of Lyndon words:

\begin{prop}\label{word_factorization}
Every word $w \in \mathcal A^*$ factorizes uniquely as a product $l_1 \cdots l_n$ where $n$ is an integer, the $l_i$ are Lyndon words and $l_1 \geq l_2 \geq ...\geq l_n$. We call this the \emph{Lyndon factorization} of $w$.
\end{prop}

\begin{proof}
We first prove unicity, by proving that in a factorization $w = l_1 \cdots l_n$ into a non-increasing product of Lyndon words, $l_n$ is the minimum nonempty suffix of $w$. 
Indeed, let $v$ be a suffix of $w$. Decompose $v$ as $yl_{k+1} \cdots l_n$ where $y$ is a nonempty suffix of $l_k$ (possibly equal to $l_k$). Then $v \geq y \geq l_k \geq l_n$.

We show existence by induction on the length  of $w$: take $l_n$ to be the smallest nonempty suffix of $w$. Then $w = w'l_n$, and $l_n$ is a Lyndon word. Moreover, a nonempty suffix of $w'$ cannot be strictly smaller than $l_n$. Indeed, if $y$ is a nonempty suffix of $w'$ such that $y < l_n$, then either $y$ is a (proper) prefix of $l_n$ or $yl_n < l_n$. The second case contradicts the definition of $l_n$. In the first case, by definition of $l_n$, we get $yl_n > l_n = yu$, whence $l_n > u$. Thus both cases contradict the definition of $l_n$: we must have $y \geq l_n$.
As a consequence, a factorisation of $w'$ satisfying the conditions of the proposition gives such a factorisation for $w$, whence the conclusion.
\end{proof}

Proposition \ref{word_factorization}  allows us to identify the abelian group $\Z\mathcal A^*$ with the symmetric algebra $S^*_{\Z}(\Lynd)$. Note that this linear identification does not preserve the ring structure, since the Lyndon factorization of a product $uv$ need not be the product of the Lyndon factorization of $u$ with that of $v$.

\subsection{The Lyndon basis of the free Lie algebra}\label{par_Lyndon_basis}

In the sequel, $V = \Z \{\mathcal A\}$ is the free abelian group generated by the alphabet $\mathcal A$. We denote by $\mathfrak LV$ the free Lie algebra on $V$ and by $TV$ the free associative algebra on $V$. Recall that their universal properties imply that $TV$ is the enveloping algebra of $\mathfrak LV$. We denote by $\iota: \mathfrak LV \rightarrow TV$ the canonical Lie morphism between them. Remark that we do not know \emph{a priori} that this map is injective (we do not assume the PBW theorem to be known).

Define an application $w \mapsto P_w$ from the set $\Lynd$ of Lyndon word on $\mathcal A$ to $\mathfrak LV$ as follows:
\begin{itemize}[topsep =0.1em, itemsep=-0.4em]
\item Take $P_a:= a \in V$ for any letter $a \in \mathcal A$;
\item If $w$ is a Lyndon word, consider its standard factorisation $w = uv$ and define $P_w$ to be $[P_u, P_v] \in \mathfrak LV$;
\end{itemize}

\begin{lem}[Standard factorization of a product of Lyndon words]\label{lem_std}
Let $u$ and $v$ be Lyndon words. Then $uv$ is a Lyndon word if and only if $u <v$. Moreover, suppose that $u < v$, and denote by $u=xy$ be the standard factorization of $u$, if $u$ is not a letter. Then the standard factorization of $uv$ is $u \cdot v$ if $u$ is a letter or $v \leq y$, and $x \cdot yv$ if $v > y$.
\end{lem}

\begin{proof}
If $uv$ is a Lyndon word, then $u < uv < v$. Conversely, suppose that $u < v$. Then either $uv < v$ or $u$ is a prefix of $v$. But in this second case, $v = uw$, and $v < w$ implies that $uv < uw = v$, so in both cases $uv < v$. Now, take a proper suffix $w$ of $uv$. If $w$ is a suffix of $v$, then $w \geq v > uv$. If not, then $w = w'v$ with $w'$ a proper suffix of~$u$. Then $u < w'$ implies $uv < w'v = w$, finishing the proof that $uv$ is a Lyndon word. 

If $u$ is a letter, then $v$ is clearly the minimal proper suffix of $uv$. Suppose that $v \leq y$. Take any proper suffix $w$ of $u$. Since $y$ is the smallest one, we have $v \leq y \leq w < wv$. As a consequence, $v$ has to be the minimal proper suffix of $uv$, whence the result in this case. If $v > y$, then $yv$ is a Lyndon word by the first part of the proof. Moreover, if $w$ is a proper suffix of $u$, then $y \leq w$, so that $yv \leq wv$. Hence $yv$ is the smallest suffix of $uv$, as needed.
\end{proof}

The following proposition and its proof are adapted from \cite[Th.\ 5.3]{Serre}. The proof is arguably the most technical one in the present appendix:

\begin{prop}\label{Lyndon_words_generate}
The $P_w$ for $w \in \Lynd$ linearly generate $\mathfrak LV$. 
\end{prop}

\begin{proof}
We only need to show that the $\Z$-module generated by the $P_w$ is a Lie subalgebra. We show that if $u$ and $v$ are Lyndon words, then $[P_u,P_v]$ is a linear combination of $P_w$, with $|w| = |u| + |v|$ and $w < \max(u,v)$, by induction on $|u|+|v|$ and on $\max(u,v)$. To begin with, if $u$ and $v$ are letters, then we can suppose that $u < v$ (otherwise, use the antisymmetry relation). Then $[P_u, P_v] = P_{uv}$, and $uv <v$.

Now, take $(u,v)$ such that $|u| + |v| > 2$, and suppose that our claim is proven for every $(u',v')$ such that $|u'|+|v'| < |u|+|v|$, or $|u'|+|v'| = |u|+|v|$ and $\max(u',v') < \max(u,v)$. Using antisymmetry if needed, we can assume that $u <v$. We then use Lemma \ref{lem_std}. When $u$ is not a letter, consider the standard factorization $u = xy$ of $u$. If $u$ is a letter or $y \geq v$, then $u \cdot v$ is the standard factorization of $uv$, whence $[P_u, P_v] = P_{uv}$, and $uv < v$, proving our claim. Suppose that $y < v$. Then:
\[[P_u, P_v] = [[P_x, P_y], P_v] = [[P_x, P_v], P_y] = [P_x, [P_y, P_v]].\]
Since $|x|, |y| < |u|$, we can use the induction hypothesis to write $[P_x, P_v]$ (resp.\ $[P_y, P_v]$) as a linear combination of $P_w$ (resp.\ $P_t$) such that $|w| = |x|+|v|$ (resp.\ $|t| = |y|+|v|$), and $w < v$ (resp. $t < v$). Then, using that $x,y < v$ (since $x < xy = u < y < v$), we can apply the induction hypothesis to $[P_w, P_y]$ (resp.\ to $[P_x, P_t]$) to prove our claim, ending the proof of the proposition.
\end{proof}

The application $w \mapsto P_w$ extends to a map from  $\mathcal A^*$ to $TV$ defined as follows:
\begin{itemize}[topsep =0.1em, itemsep=-0.4em]
\item Take $P_a:= a \in V$ for any letter $a \in \mathcal A$;
\item If $w$ is a Lyndon word, consider its standard factorization $w = uv$ and define $P_w$ to be $[P_u, P_v] \in \mathfrak LV$;
\item If $w$ is any word, consider its Lyndon factorization $w = l_1 \cdots l_n$. Define $P_w$ to be  $P_{l_1} \cdots P_{l_n}  \in TV.$
\end{itemize}

The next lemma \cite[Th. 5.1]{Reutenauer} will play a key role in what follows.

\begin{lem}\label{gr(P)=Id}
For any word $w$, the polynomial $P_w$ is the sum of $w$ and a linear combination of greater words of the same length as $w$.
\end{lem}

\begin{proof}
Remark that if $l$ is a Lyndon word and $l = uv$ with $u$ and $v$ nonempty, then $uv = l < v < vu$.

We use this to show the lemma for Lyndon words, by induction on their length. For letters, the result is obvious. Let $l$ be a Lyndon word, and consider its standard factorization $l=uv$. Then $u$ and $v$ are Lyndon word, and $u<v$ (Lemmas~\ref{lem_std_1} and~\ref{lem_std}). If the result is true for $u$ and $v$, then $P_l = [P_u, P_v]$ is a linear combination of elements of the form $[s,t] = st - ts$, where $|s| = |u|$, $|t| = |v|$, $s \geq u$ and $t \geq v$. Then $ts \geq vu > uv$, and $st \geq uv$, with equality if and only if $s = u$ and $t = v$. Thus the word $l = uv$ appears with coefficient $1$ in the decomposition of $P_l$, and $P_l - l$ is a linear combination of greater words, of the same length as $l$, which proves our claim.

Now, if $w$ is any word, consider its Lyndon factorization $w = l_1 \cdots l_n$. Then $P_w:= P_{l_1} \cdots P_{l_n}$  is a linear combination of $x_1 \cdots x_n$, where each $x_i$ is a word satisfying $|x_i| = |l_i|$ and $x_i \geq l_i$. As a consequence, $|x_1 \cdots x_n| = |l_1 \cdots l_n|$, and $x_1 \cdots x_n \geq l_1 \cdots l_n$, with equality if and only if each $x_i$ is equal to $l_i$. This last case only appears with coefficient $1$, so the lemma is proved.
\end{proof}

The above application extends to a linear map $P: \Z \mathcal A^* \rightarrow TV$.

\begin{prop}\label{Lyndon_words_free}
The application $P: \Z \mathcal A^* \rightarrow TV$ defined above is injective.
\end{prop}

\begin{proof}
Let $m$ be a linear combination of words in the kernel of $P$. Suppose that $w$ is such that no word smaller that $w$ appears in $m$. Let $\lambda$ be the coefficient of $w$ in $m$. Then by Lemma \ref{gr(P)=Id}, $\lambda$ is also the coefficient of $w$ in $P_m =0$, so it must be trivial. Thus, by induction, all coefficients of $m$ have to be trivial, whence $m = 0$ and $P$ is injective.
\end{proof}

We can now sum this up as the main result of this appendix:

\begin{theo}\label{Lyndon_basis}
The map $P$ induces a graded linear isomorphism:
\[\Z \{\Lynd\} \cong \mathfrak LV.\]
Otherwise said, the family $(P_w)_{w \in \Lynd}$ is a linear basis of $\mathfrak LV$.
\end{theo}

\begin{proof}
The $P_w$ generate $\mathfrak LV$ (Prop.\ \ref{Lyndon_words_generate}) and, since their images in $TV$ are linearly independent (Prop.\ \ref{Lyndon_words_generate}), they must be linearly independent. 
\end{proof}

\subsection{Primitive elements and the Milnor-Moore theorem}

In proving the previous result, we have only used basic linear algebra, and the combinatorics of Lyndon words. In order to convince the reader of how powerful these techniques are, we will now recover the Milnor-Moore theorem for the algebra $TV$, using not much more machinery. The only additional tools we need are coalgebra structures and primitive elements.

The free commutative ring on the free abelian group $V$ is denoted by $S^*(V)$. It is endowed with its usual Hopf algebra structure, whose coproduct is the only algebra morphism $\Delta: S^*(V) \rightarrow S^*(V) \otimes S^*(V)$ sending each element $v$ of $V$ to $v \otimes 1 + 1 \otimes v$. That is, it is the only bialgebra structure on $S^*(V)$ such that $V$ consists of primitive elements. In fact, these are the only primitive elements in $S^*(V)$ \cite[Th.\ 5.4]{Serre}:

\begin{prop}\label{Prim_of_SV}
The set of primitive elements of $S^*(V)$ is $V$.
\end{prop}

\begin{proof}
By definition of the coproduct of $S^*(V)$, the subspace $V$ is made of primitive elements. To show the converse, it is helpful to see $S^*(V)$ as the algebra $\Z[X_i]$ of polynomial in indeterminates $X_i$. Then $S^*(V) \otimes S^*(V)$ identifies with $\Z[X_i', X_i'']$, and the coproduct sends $X_i$ to $X_i' + X_i''$. Since it is an algebra morphism, it sends a polynomial $f(X_i)$ to $f(X_i' + X_i'')$. Thus primitives elements are those $f$ such that $f(X_i' + X_i'') = f(X_i') + f(X_i'')$, i.e.\ additive ones. But since we work over $\Z$, these are only the linear ones, which is the desired conclusion.
\end{proof}

The algebra $TV$ is endowed with a Hopf structure defined exactly as the one for $S^*V$: it is the unique bialgebra structure such that elements of $V$ are primitive ones. Since primitive elements are a Lie subalgebra, they contain the Lie subalgebra generated by~$V$ (which is the image $\iota(\mathfrak LV)$ of the canonical morphism $\iota: \mathfrak LV \rightarrow TV$).

Recall that Proposition \ref{word_factorization} allows us to identify $\Z\mathcal A^*$ with the symmetric algebra $S^*_{\Z}(\Lynd)$. We will show the following:

\begin{theo}[Milnor-Moore]\label{S(Lynd)_and_TV}
The application $P: S^*_{\Z}(\Lynd) \rightarrow TV$ defined above in \cref{par_Lyndon_basis} is an isomorphism of coalgebras.
\end{theo}

\begin{proof}
Injectivity has already been shown (Prop. \ref{Lyndon_words_free}). Let us first prove surjectivity.
Let $p \neq 0$ be a homogeneous element of $TV$. Let $w$ be the smallest monomial appearing in $p$, with coefficient $\lambda$. Then $p - \lambda P_w$ is homogeneous and contains only monomials greater than $P_w$ (see Lemma \ref{gr(P)=Id}). By repeating this process, we can write $p$ as a linear combination of $P_w$. Indeed, this process stops, since we consider only the finite set of words of fixed length (equal to the degree of $p$) whose letters appear in some monomial of $p$.

We are left to show that the application $P: w \mapsto P_w$ preserves the coproduct. We first remark that if $l$ is a Lyndon word, then $l$ is primitive in $S^*_{\Z}(\Lynd)$, and $P_l \in \iota(\mathfrak LV)$ is primitive in $TV$. For any word $w$, consider its Lyndon factorization $w = l_1 \cdots l_n$. Then we can write:
\begin{align*}
\Delta(P_w) &= \Delta(P_{l_1}) \cdots  \Delta(P_{l_n}) \\ &= 
(P_{l_1}\otimes 1 + 1 \otimes P_{l_1}) \cdots  (P_{l_n}\otimes 1 + 1 \otimes P_{l_n}) \\ &=
\sum_{\underline n = X \sqcup Y} P_{l_{x_1}} \cdots P_{l_{x_p}} \otimes P_{l_{y_1}} \cdots P_{l_{y_q}},
\end{align*}
where the sum is over all partitions of the set $\underline n = \{1, ...,n\}$ into subsets $X = \{x_1 < \cdots < x_p\}$ and $Y = \{y_1 < \cdots < y_q\}$. As a consequence:
\begin{align*}
\Delta(P_w) &= \sum_{\underline n = X \sqcup Y} P_{l_{x_1} \cdots l_{x_p}} \otimes P_{l_{y_1} \cdots l_{y_q}} \\ &=
(P \otimes P)\left( \sum_{\underline n = X \sqcup Y} l_{x_1} \cdots l_{x_p} \otimes l_{y_1} \cdots l_{y_q} \right) \\ &= 
(P \otimes P)(\Delta(l_1) \cdots \Delta(l_n)) = (P \otimes P)(\Delta(w)),
\end{align*}
which ends the proof of the theorem.
\end{proof}

\begin{cor}\label{Prim(TV)}
The canonical map $\iota: \mathfrak LV \rightarrow TV$ identify $\mathfrak LV$ with the Lie algebra of primitive elements in $TV$.
\end{cor}

\begin{proof}
Thanks to Theorem \ref{S(Lynd)_and_TV} and Theorem \ref{Lyndon_basis}, this map identifies with $\Z \{\Lynd\} \rightarrow  S^*_{\Z}(\Lynd)$. But Proposition \ref{Prim_of_SV} ensures that the set of primitive elements of the coalgebra $S^*_{\Z}(\Lynd)$  is exactly $\Z \{\Lynd\}$, whence the result.
\end{proof}

\begin{rmq}
Neither our identification of the free abelian group $\Z \{\Lynd\}$ with the primitives of $TV$ nor our proof of Theorem \ref{S(Lynd)_and_TV} requires the use of the fact that Lyndon words generate $\mathfrak LV$ (Prop.~\ref{Lyndon_words_generate}) : we only used that they are linearly independent (Prop.~\ref{Lyndon_words_free}) for that. The full strength of Th.~\ref{Lyndon_basis} is only used to see that $P : \mathbb Z \mathcal A^* \hookrightarrow TV$ coincides with $\iota: \mathfrak LV \rightarrow TV$ (whence Cor.~\ref{Prim(TV)}).
\end{rmq}

\subsection{Linear trees}

The free Lie algebra can be seen as a quotient of the free abelian group $\Z \mathcal M(\mathcal A)$ on the free magma $\mathcal M(\mathcal A)$ on $\mathcal A$ by antisymmetry and the Jacobi identity. Elements of the free magma can be seen as parenthesized words in $\mathcal A$, or as finite rooted planar binary trees, whose leaves are indexed by elements of $\mathcal A$. The images of elements of the free magma in $\mathfrak LV$ are called \emph{Lie monomials}. 

Lyndon words encode a family of rooted binary trees whose leaves are indexed by letters. Precisely, if $w$ is a Lyndon word, the tree $T(w)$ is just one leaf indexed by $w$, if $w$ is a letter. If not, take the standard factorization $w = uv$. Then $T(w)$ is given by a root, a left son $T(u)$ and a right son $T(v)$:
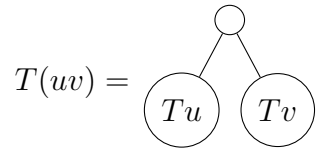
\[T(uv) =
\begin{gathered} 
\begin{forest}
for tree={circle,draw}
    [[$Tu$][$Tv$]]
  \end{forest}
\end{gathered}
\]
The Lyndon basis of the free Lie algebras are the Lie monomials $P_w$ obtained from such trees by interpreting nodes as Lie brackets. We call these \emph{Lyndon monomials}

One can consider another family of Lie monomials, called \emph{linear Lie monomials}, given by linear trees, that is, monomials which are letters or of the form $[y_1, ..., y_n]$ ($= [y_1, [y_2, [...[y_{n-1}, y_n]...]]$). It is easy to see, by induction, using the Jacobi identity, that these generate $\mathfrak LV$. In fact, the Jacobi identity can be written as:
\begin{equation}\label{Jacobi_trees}
\begin{gathered}
\begin{forest}
for tree={circle,draw}
    [[[A][B]][C]]
  \end{forest}
\end{gathered}
\ \ \  =\ \ \ 
\begin{gathered}
\begin{forest}
for tree={circle,draw}
    [[A][[B][C]]]
  \end{forest}
\end{gathered}
\ - \ \ 
\begin{gathered}
\begin{forest}
for tree={circle,draw}
    [[B][[A][C]]]
  \end{forest}
\end{gathered}.
\end{equation}
Using this as a rewriting rule (from left to write), one can write any tree (that is, any Lie monomial) as a linear combination of trees whose left son is a leaf. Applying the induction hypothesis to the right sons, one gets a linear combination of linear trees.

There are $n!$ linear Lie monomials in degree $n$, which is clearly strictly greater than the number of Lyndon words of length $n$, so they must be linearly dependent. It is the need to control this redundancy that leads to consider Lyndon words (or, more generally, Hall sets).

\begin{lem}\label{lem_letters}
Any Lie monomial is a linear combination of linear Lie monomials with the same letters (counted with repetitions). Also, it is a linear combination of Lyndon monomials with the same letters. 
\end{lem}

\begin{proof}
The first part follows from the rewriting process that we have just described. The second one is a bit trickier: although we know that a decomposition into a linear combination of Lyndon monomials \emph{exists} (Prop \ref{Lyndon_words_generate}), we did not give an algorithm to compute it. However, we can use a homogeneity argument, as follows: $\Z \mathcal M(\mathcal A)$ is $\mathbb N \{\mathcal A\}$-graded, the degree of an element of the free magma $\mathcal M(\mathcal A)$ being its image in the free commutative monoid $\mathbb N \{\mathcal A\}$ (which counts the number of appearance of each letter in a given non-associative word). Moreover, the antisymmetry and the Jacobi relations are homogeneous with respect to this degree, so that the quotient $\mathfrak L[\mathcal A]$ is again a graded abelian group with respect to this degree. As a consequence, if we write a Lie monomial $m$ of degree $d$ as a linear combination of Lyndon monomials, taking the homogeneous component of degree $d$ results in an expression of $m$ as a linear combination of Lyndon monomials of degree $d \in \mathbb N \{\mathcal A\}$, as claimed.
\end{proof}

Remark that the expression of $m$ obtained in the proof by taking the homogeneous component must in fact must be the same as the first one, because of Theorem \ref{Lyndon_basis}.

\medskip

Linear trees can be used to define a basis of the reduced free Lie ring $R\mathfrak L[n]$ which could be used to replace the Lyndon basis in all our work (this is in fact the point of view used in \cite{Meilhan-Yasuhara}):

\begin{lem}
For all integer $k \geq 2$, a basis of $R\mathfrak L[n]_k$ is given by Lie monomials which are letters or of the form $[y_{i_1}, ..., y_{i_k}]$ where the $i_j \leq n$ are pairwise distinct and satisfy $i_k = \max\limits_j (i_j)$.
\end{lem}

\begin{proof}
Using antisymmetry, one sees that up to a sign, any Lie monomial without repetition is equal to a Lie monomial with the same letter where the right-most factor (the right-most leaf of the corresponding tree) bears the maximal index. Then we can use the re-writing rule \eqref{Jacobi_trees} to get a linear combination of linear trees, and the right-most leaf stays the same throughout the process, as does the set of letters used. This shows that Lie monomials of the form described in the lemma generate the abelian  group $R\mathfrak L[n]_k$. Moreover, there are $(k-1)! \binom{n}{k}$ such monomials, which is already known to be the rank of $R\mathfrak L[n]_k$ (Prop. \ref{rk_of_RL}), hence this family must be a basis of $R\mathfrak L[n]_k$.
\end{proof}

\end{subappendices}

\bibliographystyle{alpha}
\bibliography{Ref_homotopy_inv}

\end{document}